\documentclass[12pt, reqno]{amsart}

\usepackage{amssymb,mathtools}
\usepackage{amsmath}
\numberwithin{equation}{section}
\usepackage{mathrsfs}
\usepackage{amsfonts}
\usepackage{xcolor,float}
\usepackage{vmargin}
\usepackage{amsthm}
\usepackage{graphicx}
\usepackage[colorlinks=true, linkcolor=blue, citecolor=red, urlcolor=magenta]{hyperref}
\usepackage{placeins}
\usepackage{lineno}
\usepackage[colorinlistoftodos,prependcaption,textsize=tiny]{todonotes}
\usepackage[font=small,skip=5pt]{caption}
\usepackage{float}

\newtheorem{theorem}{Theorem}[section]
\newtheorem{lemma}[theorem]{Lemma}

\theoremstyle{definition}
\newtheorem{definition}[theorem]{Definition}
\newtheorem{example}[theorem]{Example}

\newtheorem{corollary}[theorem]{Corollary}
\theoremstyle{remark}

\numberwithin{equation}{section}
\newcommand{\Zb}{\mathbb{Z}}
\newcommand{\R}{\mathbb{R}}
\newcommand{\N}{\mathbb{N}}
\newcommand{\Lp}{L^p(I, \rho)}
\newcommand{\X}{\mathbf{x}}

\newcommand{\LP}{L^p(I^d, \rho)}
\usepackage{color}
\usepackage{lineno}

\begin{document}
\title[Neural Network operators with respect to arbitrary measures]{Approximation by Neural Network operators in\\ $L^p$ spaces associated with an arbitrary measure}

\author{Nitin Bartwal}

\address{Department of Mathematics, Indian Institute of Technology Madras, Chennai-600036, Tamil Nadu, India}
\email{ma22d017@smail.iitm.ac.in, bartwalsinghnitin@gmail.com}

\author{A. Sathish Kumar}

\address{Department of Mathematics, Indian Institute of Technology Madras, Chennai-600036, Tamil Nadu, India}
\email{sathishkumar@iitm.ac.in, mathsatish9@gmail.com}


\begin{abstract}

   In this paper, we investigate the approximation behavior of both one and multidimensional neural network (NN) inspired operators for functions in $L^p(I^d,\rho)$, where $1\leq p<\infty$, associated with a general measure $\rho$ defined over a hypercube $I^d$. First, we prove the uniform approximation for a continuous function and the $L^p$ approximation theorem by the NN operators in one and multidimensional settings. In addition, we also obtain the $L^p$ error bounds in terms of $\mathcal{K}$-functionals for these neural network operators. Finally, we consider logistic and tangent hyperbolic activation functions and verify the hypothesis of the theorems. We also show the implementation of continuous and integrable functions approximation by NN operators with respect to the Lebesgue and Jacobi measures defined on $[0,1]\times[0,1]$ with logistic and tangent hyperbolic activation functions. 
\end{abstract}

\maketitle

\section{Introduction}

 A neural network is a mathematical framework modeled after the structural and functional principles of the human brain, aiming to replicate the cognitive processes by which humans interpret information and learn from past interactions. This learning mechanism is implemented through multiple layers of interconnected units called  neurons that process input data using successive applications of affine transformations followed by nonlinear activation functions.
Let $x\in \mathbb{R}^d$ and $d\in\mathbb{N}$. Then, the feed-forward neural network (FNNs) with one hidden layer is defined by 
\begin{eqnarray*}
N_{n}(x)=\displaystyle\sum_{\ell=0}^{n}c_{\ell}\sigma(\langle\alpha_{\ell}.x\rangle+\beta_{\ell}),
\end{eqnarray*}
where $0\leq\ell\leq n,$ $\beta_{\ell}\in\mathbb{R}$ are thresholds and $\alpha_{\ell}\in \mathbb{R}^d$ are connection weights and $c_{\ell}\in \mathbb{R}$ are the coefficients. It is well known that FNNs with one hidden layer and non-polynomial activation function can approximate any continuous function uniformly on compact subsets of $\R^d$ if given a sufficient number of neurons \cite{Cybenko}. Further, the approximation of measurable functions by these neural networks was analyzed in \cite{Hornik}.\par
  When the underlying function is known, a neural network can be constructed directly rather than learned through repeated training. By incorporating the structure of the function into the network architecture and weights, accurate approximation can be achieved without iterative optimization. This reduces computational effort, improves numerical stability, and makes the resulting model easier to interpret. 
 In \cite{carda}, Cardaliaguet and Euvrard analyzed the approximation properties of both functions and their derivatives using the feed forward neural network. Inspired by this work, Anastassiou studied the approximation of continuous functions and their rate of convergence of neural network operators in \cite{Anas1}. Further, he analyzed the approximation behavior NN operators using different activation functions in one and multidimensional settings, see \cite{Anas2, Anas3, GAA1, GAA2, GAA3} and the references therein. 
Furthermore, the point-wise, uniform convergence results and the order of convergence of the NN operators were proved by Costarelli and Spigler in \cite{Costar1, Costar2}. The approximation behavior of Kantorovich neural network operators were analyzed in different settings, see \cite{Cosr, Cosrv1, Cosrv2, Cosrv3} and the references therein. Approximation by NN-operators have been studied widely by several authors in different directions, see \cite{ASB, PNA, VEI, CMR, LCD, COM, HLG, CHE}. Recently, Costarelli \cite{Cost1} estimated the approximation error for the NN operators in terms of the averaged modulus of smoothness in the settings of the $L^{p}$ spaces corresponding to the Lebesgue measure. We extend the study of the approximation properties of NN operators in $L^p$ spaces, where $1\leq p\leq \infty$, associated with a general measure. This measure is defined on a $d$-dimensional hypercube and assumed to satisfy a  support condition. Weighted Lebesgue spaces are particular instances within this broader class of function spaces and can be used in image analysis. 

More specifically, an image can be represented as an element of a weighted Lebesgue space, which provides a functional analytic framework for image analysis. Formally, a grayscale image can be viewed as a measurable function 
 $$f:\Omega\subset\R^2\to \R,$$ where $\Omega$ is the domain of the image and $f(x,y)$ gives the pixel intensity at point $(x,y).$ In a weighted space, the function $f(x,y)$ is equipped with the norm
 $$\|f\|_{L^p}:=\left(\int_{\Omega}|f(x,y)|^p \,w(x,y)dx\,dy\right)^\frac{1}{p},$$ where $w(x,y)$ is a positive weight function. The weight alters the contribution of different regions of the image to the overall measurement, allowing one to emphasize or de-emphasize specific areas. This is useful in applications such as feature detection, where regions of interest may be prioritized, or in noise modeling, where uncertain areas can be down-weighted. Thus, treating an image as an element of a weighted Lebesgue space not only embeds it in a rigorous mathematical structure but also provides flexibility for adapting analysis to the characteristics of the image.\\
Due to the importance of weighted norm spaces in image analysis, we study the approximation properties of the NN inspired operators for functions belonging to $L^p(I^d,\rho)$, where $\rho$ is any measure on the hypercube satisfying some support condition. This work is inspired by the work of Berdysheva and her collaborators. Berdysheva and Jetter \cite{Elena1} initiated the study of Bernstein-Durrmeyer operator with respect to arbitrary measure on $d$-dimensional simplex $S^d.$ Further, she proved the uniform convergence of these operators for continuous functions by assuming the strict positivity of measure $\rho$ on the simplex in \cite{Elena2}. Furthermore, by relaxing the conditions on the support of measure $\rho$, she proved the point-wise and uniform approximation results for these operators in \cite{Elena3}. Before analyzing the convergence behavior of the neural network operators with respect to an arbitrary measure on $I^d,$ we recall the following notations and basic definitions. 
 
\subsection{Notations and Preliminaries}
We consider the following notations and preliminaries, which shall be used throughout this paper.

Let $I^d:=[0,1]^d:=\{\mathbf{x}=(x_1,x_2,\ldots,x_d): 0\leq x_1,x_2,\ldots,x_d\leq 1\}$ be the hypercube of dimension $d$ in $\R^d$. 
Let $\beta$ be a multi-index such that
$$\beta=(k_1,k_2,\ldots,k_d), \ \ \mbox{and} \ \ \frac{\beta}{n}=\left(\frac{k_1}{n}, \frac{k_2}{n},\ldots,\frac{k_d}{n}\right),$$ where $0\leq k_i\leq n$ for $i\in \{1,\ldots, d\}.$ \\

Consider a partition of $I^d$ into $2^d$ subsets. For a subset $A\subseteq\{1,2,\ldots,d\},$ define the set $I_A\subseteq I^d$ as follows: $$I_A=\{\mathbf{x}:=(x_1, x_2,\ldots,x_d)\in I^d: x_i\geq\frac{1}{2},\, \forall i\in A \}.$$
Let $\eta>0.$ For $\frac{\beta}{n}=(\frac{k_1}{n},\frac{k_2}{n},\ldots,\frac{k_d}{n})\in I_A$, we define the set $V_{A,\,\delta}(\frac{\beta}{n})$ by
$$V_{A,\,\delta}\left(\frac{\beta}{n}\right)=\prod_{i=1}^n\left[\frac{k_i}{n}-\eta\,\mathbf{1}_{A}(i),\, \frac{k_i}{n}+\eta\,\mathbf{1}_{A^c}(i)\right],$$ where $\mathbf{1}_{A}$ denotes the indicator function of $A.$ It is easy to verify that if $\frac{\beta}{n}\in I_A$, then $V_{A,\,\eta}\left(\frac{\beta}{n}\right)\subset I^d$ for $0<\eta<\frac{1}{2}.$\\

Furthermore, for a given $\frac{\beta}{n}=(\frac{k_1}{n},\frac{k_2}{n},\ldots,\frac{k_d}{n}),$ we also define the set  $U_\eta({\frac{\beta}{n})}$ by
 $$U_\eta\left({\frac{\beta}{n}}\right):=\prod_{i=1}^d\left(\frac{k_i}{n}-\eta,\,\frac{k_i}{n}+\eta\right).$$

\begin{definition}
A bounded Borel measure $\rho$ is said to be strictly positive on $I^d$ if 
$\rho(A\cap I^d)>0$
for every open set of $A\subseteq\R^d$ such that $A\cap I^d\neq\emptyset.$ 
\end{definition}

Now, we define the $L^p(I^d,\rho)$ space. Let $1\leq p< \infty$. We denote by $L^p(I^d,\rho)$ the space of all real-valued measurable functions on $I^d$ such that
\begin{equation*}
\int_{I^d}|f(x)|^pd\rho(x)<\infty.
\end{equation*}
The corresponding norm on $L^p(I^d, \rho)$ is given by
$$\|f\|_{L^{p}(I^d,\rho)}:=\left(\int_{I^d}|f(x)|^pd\rho(x)\right)^{\frac{1}{p}}.$$
The space $L^{\infty}(I^d,\rho)$ is the set of all essentially bounded functions on the hypercube $I^d.$ The corresponding norm on $L^{\infty}(I^d,\rho)$ is given by 
$$\|f\|_{L^{\infty}(I^d,\rho)}:=\displaystyle{{\text{ess} \sup}_{x\in I^d}}|f(x)|.$$
We denote by $C(I^d)$ the space of all continuous functions on $I^d$ and their norm is defined by 
$$\|f\|_\infty:=\displaystyle{\sup_{x\in I^d}}|f(x)|.$$

Now, we recall some basic definitions and properties of the sigmoidal function $\sigma.$

\begin{definition}
A sigmoidal function $\sigma$ is a measurable function with 
$$\lim_{x\to-\infty}\sigma(x)=0 \,\, \mbox{and} \lim_{x\to+\infty}\sigma(x)=1.$$ 
\end{definition}

Throughout this article $\sigma$ is assumed to be a non-decreasing function satisfying the following assumptions unless stated otherwise: 
\begin{enumerate}
    \item[$(A_1)$] $\sigma(x)-\frac{1}{2}$ is an odd function. 
    \item[$(A_2)$] $\sigma\in C^2(\R)$ is concave for $x\geq0.$ \
     \item[$(A_3)$] $\sigma(x)=\mathcal{O}(|x|^{-\beta})$ as $x\to-\infty$ for some $\beta>1.$ 
\end{enumerate} 

\begin{definition}
    For the sigmoidal function $\sigma,$ we define the density or kernel function $\phi_\sigma$ as follows: 
    \begin{align}
\phi_\sigma(x):= \frac{1}{2}(\sigma(x+1)-\sigma(x-1)).\label{kernel} 
\end{align}
\end{definition}

We now list some well-known properties of the kernel $\phi_\sigma$ that will be used throughout this article. For more details and proofs of these properties, one can refer to \cite{Costar1}.
\begin{enumerate}
\item[$(1)$] $\phi_\sigma(x)$ is a non negative function. 
\item[$(2)$] $\phi_\sigma(x)$ is non decreasing for $x<0$ and non increasing for $x\geq0.$
\item[$(3)$] $\phi_\sigma(x)=\mathcal{O}(|x|^{-\beta})$ as $x\to\pm\infty.$ 
\item[$(4)$] For every $x\in \R$, we have 
$$\sum_{k\in\Zb}\phi_\sigma(x-k)=1.$$
\item[$(5)$] Let $x\in I$ and $n\in{\N}$. Then, we have `
$$\sum_{k=0}^n\phi_\sigma(nx-k)\geq\phi_\sigma(1)>0.$$
\end{enumerate}
\begin{definition}
    The $r^{th}$ order discrete absolute moment of $\phi_\sigma(x)$ is defined as $$M_r(\phi_\sigma):=\sup_{x\in\mathbb{R}}\sum_{k\in\Zb  }|x-k|^r\phi_\sigma(x-k).$$ 
\end{definition}
Under the assumption $(A_3)$ on $\sigma$ (see \cite{Costar1}), we have $$M_r(\phi_\sigma)<+\infty, \  \mbox{for} \ 0\leq r<\beta-1.$$

To obtain quantitative estimates for the rate of convergence of $L^p$ approximation, we employ the $\mathcal{K-}$ functionals.   
\begin{definition}  The $\mathcal{K}-$functional for a function $f\in \LP$ is defined as follows:
$$\mathcal{K}(f,t)_p:=\inf_{g\in W^{1, \infty}(I^d,\rho)}\{\|f-g\|_{\LP}+t\|g\|_{1,\infty}\},$$  
where the associated Sobolev space $W^{1,\infty}(I^d, \rho)$ is defined by $$W^{1,\infty}(I^d, \rho):=\left\{g: g,\,\frac{\partial g}{\partial x_j}\in L^{\infty}(I^d,\rho)\, \,  \,\, 1\leq j\leq d\right\},$$ and $\|g\|_{1,\infty}$ is a semi-norm on $W^{1,\infty}(I^d,\rho)$, and is given by 
$$\|g\|_{1,\infty}:= \sum_{j=1}^d \left\|\frac{\partial g}{\partial x_j}\right\|_{L^\infty(I^d,\rho)}.$$ It is important to note that the partial derivatives here are considered in the weak sense. 
\end{definition}

Let $\Phi_\sigma:\R^d\to\R$ be such that $$\Phi_\sigma(x_1, x_2,\ldots,x_d):=\prod_{i=1}^d\phi_\sigma(x_i),$$ where $\phi_\sigma$ is the usual kernel defined in (\ref{kernel}). Now, we define the multivariate NN operator with respect to the measure $\rho$ for $f:I^d\to\R$, where $f$ is a suitable function that depends on the space under consideration.
\begin{definition}
Let $\rho$ be a non negative bounded Borel measure on $I^d$ and $1\leq p \leq \infty$. For $f\in \LP,\, n\in \N^+.$ The multivariate Neural Network operators with respect to measure $\rho$ is defined by \begin{equation}
 S^{\rho}_nf(\X)=\displaystyle{\dfrac{\sum_{k_1=0}^n\sum_{k_2=0}^n\ldots\sum_{k_d=0}^nc_{n,\beta}\,\Phi_\sigma(nx_1-k_1,nx_2-k_2,\ldots,nx_d-k_d)}{\sum_{k_1=0}^n\sum_{k_2=0}^n\ldots\sum_{k_d=0}^n\Phi_\sigma(nx_1-k_1,nx_2-k_2,\ldots,nx_d-k_d)}}, \label{mopdef} \end{equation} where the coefficient $c_{n,\beta}$ is given by 
$$c_{n,\beta}:=\dfrac{\int_{I^d}f(t)\,\Phi_\sigma(nt_1-k_1,nt_2-k_2,\ldots,nt_d-k_d)\,d\rho(t)}{\int_{I^d}\Phi_\sigma(nt_1-k_1,nt_2-k_2,\ldots,nt_d-k_d)\,d\rho(t)}.$$ \end{definition}

It is easy to see that the operator (\ref{mopdef}) is well defined for all $f\in L^{\infty}(I^d,\rho).$ Indeed, we have
\begin{equation*}
    |S^{\rho}_nf(x)|\leq \max_{\frac{\beta}{n} \in I^d}|c_{n,\beta}| \leq \dfrac{\int_{I^d}|f(t)|\,\Phi_\sigma(nt_1-k_1,nt_2-k_2,\ldots,nt_d-k_d)\,d\rho(t)}{\int_{I^d}\Phi_\sigma(nt_1-k_1,nt_2-k_2,\ldots,nt_d-k_d)\,d\rho(t)}
    \leq \|f\|_{\infty}.
    \end{equation*}

This paper is structured as follows. In Section \ref{Uni}, we consider the univariate version of the operator (\ref{mopdef}) and show that it converges uniformly for all continuous functions on $I.$ In addition, we also prove that the family of operators $\{S_n^\rho\}_{n\in\N}$ is uniformly bounded in $\Lp$, and using the density of the continuous function, we obtain the $\Lp$ norm convergence of the operator. Further, we also get $L^p$ error bounds for the operator in terms of $\mathcal{K}$-functionals. In Section \ref{multi}, we extend the approximation results of section \ref{Uni} to the multidimensional setting, by taking neural network operator defined on a hypercube. In Section \ref{example}, we focus on some specific sigmoidal functions and verify the assumptions of the theorems to validate the proposed theory. In addition, we approximate particular continuous and integrable functions by NN operators with respect to the Lebesgue and Jacobi measures defined on $[0,1]\times[0,1]$ with logistic and tangent hyperbolic activation functions.

\section{Univariate Neural Network operators with respect to arbitrary measures} \label{Uni}
In this section, we consider the univariate version of the operator (\ref{mopdef}), and we derive the uniform approximation and $L^p$ error bounds in terms of $\mathcal{K}$-functionals.
The uniform approximation results were inspired by the work of Berdysheva and her collaborators (see \cite{Elena1}, \cite{Elena2}, \cite{Elena3}), while the $L^p$-convergence theorems were influenced by Li \cite{BZli} on the Bernstein–Durrmeyer operators. Building on their ideas, we establish the uniform and 
convergence theorems for the NN operators associated with an arbitrary measure.
These results will be used to get the approximation results of multidimensional Neural Network operators.  We denote the interval $[0,1]$ by $I.$

Before delving into the analysis,  we briefly recall some basic definitions and results which will be useful to derive the uniform convergence of univariate Neural Network operators. 
\begin{definition}
For $f\in C(I),$ the neural network operator $F_n$ is defined as follows: 
$$ F_nf (x): =\dfrac{\displaystyle{\sum_{k=0}^n}f\left(\frac{k}{n}\right)\phi_\sigma(nx-k)}{\displaystyle{\sum_{k=0}^n}\phi_\sigma(nx-k)}, \quad\quad x\in I.$$ 
\end{definition}

We recall the following convergence result from \cite{Costar1}. 
\begin{theorem} \label{opuni}
    Let $f:[0,1]\to\R$ be a bounded function. Then
    $$\lim_{n\to\infty}F_nf(x)=f(x)$$ at any point $x\in[0,1]$ of continuity of $f.$ Moreover, if $f$ is continuous on $[0,1]$ then we have
    $$\lim_{n\to\infty}\|F_nf-f\|_{\infty}=0.$$
\end{theorem}
Now we define the univariate version of the operator (\ref{mopdef}) and discuss their approximation properties.
\begin{definition}
Let $\rho$ be a strictly positive bounded Borel measure on $I$. We define the neural network operator $S^{\rho}_{n}$ with respect to the measure $\rho$ for $f\in C(I)$ as follows:
\[S^{\rho}_{n}f(x):=\dfrac{\displaystyle{\sum_{k=0}^{n}}{c_{n,k}\,\phi_\sigma(nx-k)}}{\displaystyle{\sum_{k=0}^{n}}\phi_\sigma(nx-k)}\quad\quad x\in I,\label{opdefuni}\] where the coefficient $c_{n,k}$ is given by 
$$c_{n,k}:=\dfrac{\int_{0}^{1}f(t)\phi_\sigma(nt-k)\,d\rho(t)} {\int_{0}^{1}\phi_\sigma(nt-k)\,d\rho(t)}.$$ 
 It is easy to see that $S^{\rho}_{n}$ is a positive linear operator and reproduces constant functions.
\end{definition}
    
First, we prove the following lemma which will be useful in proving the uniform approximation of the continuous function on $I$ using a univariate NN operator. 

\begin{lemma} \label{mmpos}
 Let $\rho$ be a strictly positive bounded Borel measure on $I^d$. Then for each  $\delta>0$ and  $A\subseteq \{1,2,\ldots,d\}$, we have 
  $$\inf_{\frac{\mathbf{\beta}}{n}\in I_A}\rho\left(V_{A,\delta}\left(\frac{\beta}{n}\right)\right) >0,$$ where $\dfrac{\beta}{n}=\left(\dfrac{k_1}{n},\dfrac{k_2}{n},\ldots,\dfrac{k_d}{n}\right)$, $n\in\N$ and $0\leq k_i\leq n.$ 
  \end{lemma}
  
  \begin{proof} For $\delta \leq\frac{1}{2}$, we know that $V_{A,\,\delta}\left(\frac{\beta}{n}\right)\subset I^d.$ Take $N\in\N$ such that $\frac{1}{N}<\frac{\delta}{2},$ and cover $I^d$ with cubes $\Delta_{\frac{\alpha}{N}},$ where $\Delta_{\frac{\alpha}{N}}=\prod_{i=1}^d[\frac{\alpha_i}{N}, \frac{\alpha_i+1}{N}]$ and $\alpha=(\alpha_1, \alpha_2,\ldots,\alpha_d)$ and $\alpha_i\in\{0, 1, 2,\ldots, N-1\}.$ It is easy to check that $\Delta_\alpha \subseteq I^d.$ Since $\rho$ is a strictly positive Borel measure on $I^d, $ we have $\min_{\alpha}\rho(\Delta_{\frac{\alpha}{N}})>0.$  
  The set $V_{A,\,\delta}\left(\frac{\beta}{n}\right)$ contains $V_{A,\,\frac{\delta}{2}}\left(\frac{\beta}{n}\right)$, which contains $\Delta_{\frac{\alpha}{N}}$ for some $\alpha,$ where $\alpha=(\alpha_1, \alpha_2.\ldots,\alpha_d)$ and $\alpha_i\in \{0,1,\ldots,N-1\}.$  Hence, for all $\frac{\beta}{n}\in I_A,$ we have
  $$\rho\left(V_{A,\,\delta}\left(\frac{\beta}{n}\right)\right)\geq\min_\alpha \rho(\Delta_{\frac{\alpha}{n}})>0.$$ 
  For $\delta>\frac{1}{2}$ the statement follows from the monotone property of the measure $\rho$
  and the case when $\delta<\frac{1}{2}.$
  

  \end{proof}

  Now, we prove the uniform convergence of the operator $S^{\rho}_{n}$ for continuous functions on $I.$

\begin{theorem} \label{uniapp}
Let $\rho$ be a strictly positive bounded Borel measure on $I.$  Suppose for $0<\delta<1,$ either
\begin{equation}
  \lim_{n\to\infty} \dfrac{\max\{\phi_\sigma(nt-k): t\in [0,1]\setminus (\frac{k}{n}-\delta, \frac{k}{n}+\delta) \}}{\min\{\phi_\sigma(nt-k): t\in[\frac{k}{n}, \frac{k}{n}+\delta^2]\}}= 0 
\end{equation} or \begin{equation}
  \lim_{n\to\infty} \dfrac{\max\{\phi_\sigma(nt-k): t\in [0,1]\setminus (\frac{k}{n}-\delta, \frac{k}{n}+\delta) \}}{\min\{\phi_\sigma(nt-k): t\in[\frac{k}{n}-\delta^2, \frac{k}{n}]\}}= 0. 
\end{equation}
Then for every $f\in C(I)$, we have 
$$\lim_{n\to\infty}\|S^{\rho}_nf-f\|_{\infty}=0.$$
\end{theorem}

\begin{proof} For all $x\in I$, we have
\begin{eqnarray*}
|S^{\rho}_nf(x)-f(x)|
&\leq& |S^{\rho}_nf(x)-F_nf(x)| + |F_nf(x)-f(x)|\\
&:=& I_1 + I_2.
\end{eqnarray*}
    By Theorem \ref{opuni}, we have $I_2\to 0$ uniformly as $n\to\infty.$ So we only need to estimate the term $I_1.$ We have 
\begin{eqnarray*}
I_1 
&\leq& \dfrac{\displaystyle{\sum_{k=0}^{n}}\left|c_{n,k}-f\left(\frac{k}{n}\right)\right|\,\phi_\sigma(nx-k)}{\sum_{k=0}^{n}\phi_\sigma(nx-k)} \\
&\leq& \max_{k=0,1,\ldots,n}\left|c_{n,k}-f\left(\frac{k}{n}\right)\right|.
\end{eqnarray*}
It is enough to show that $\max_{k=0,1,\ldots,n}\left|c_{n,k}-f\left(\frac{k}{n}\right)\right|\to0$ as $n\to\infty.$ 

Since $f$ is continuous on $I,$ so for $\forall\,\epsilon>0,$ $\exists \,\, 0<\delta<\frac{1}{2}$ such that $|f(x)-f(y)|<\epsilon$ whenever $|x-y|<\delta$. We divide the proof into two cases: \\
\textbf{Case 1:} $\frac{k}{n}\leq\frac{1}{2}$
\begin{eqnarray*}
\left|c_{n,k}-f\left(\frac{k}n{}\right)\right|
&\leq&\dfrac{\displaystyle\int_{0}^{1}\left|f(t)-f\left(\frac{k}{n}\right)\right|\phi_\sigma(nt-k)\,d\rho(t)}{\int_{0}^1\phi_\sigma(nt-k)\,d\rho(t)}\\ 
&=& \dfrac{\displaystyle{\int_{(\frac{k}{n}-\delta, \frac{k}{n}+\delta)\cap [0,1]}}\left|f(t)-f\left(\frac{k}{n}\right)\right|\phi_\sigma(nt-k)\,d\rho(t)}{\int_{0}^1\phi_\sigma(nt-k)\,d\rho(t)} \\&&+\dfrac{\displaystyle{\int_{[0,1]\setminus(\frac{k}{n}-\delta, \frac{k}{n}+\delta)}}\left|f(t)-f\left(\frac{k}{n}\right)\right|\phi_\sigma(nt-k)\,d\rho(t)}{\int_{0}^1\phi_\sigma(nt-k)\,d\rho(t)}. 
\end{eqnarray*}
Using the uniform continuity and boundedness of $f$ in $I$, we get 
 \begin{eqnarray*}
\left|c_{n,k}-f\left(\frac{k}n{}\right)\right| &\leq& \epsilon + 2 M \dfrac{\displaystyle{\int_{[0,1]\setminus(\frac{k}{n}-\delta, \frac{k}{n}+\delta)}}\phi_\sigma(nt-k)\,d\rho(t)}{\int_{0}^1\phi_\sigma(nt-k)\,d\rho(t)} \\
 &\leq& \epsilon + 2M\,\dfrac{\max\{\phi_\sigma(nt-k): t\in [0,1]\setminus (\frac{k}{n}-\delta, \frac{k}{n}+\delta) \}}{\min\{\phi_\sigma(nt-k): t\in[\frac{k}{n}, \frac{k}{n}+\delta^2]\}} \ \ \dfrac{\rho([0,1]}{\rho([\frac{k}{n}, \frac{k}{n}+\delta^2])} \\
 &\leq& \epsilon+ 2M\,\dfrac{\max\{\phi_\sigma(nt-k): t\in [0,1] \setminus(\frac{k}{n}-\delta, \frac{k}{n}+\delta) \}}{\min\{\phi_\sigma(nt-k): t\in[\frac{k}{n}, \frac{k}{n}+\delta^2]\}} \ \ \dfrac{\rho(I)}{\rho([\frac{k}{n}, \frac{k}{n}+\delta^2])}. 
    \end{eqnarray*}
By Lemma \ref{mmpos}, we have $\rho([\frac{k}{n}, \frac{k}{n}+\delta^2])>0.$ Further, by the hypothesis of the theorem, we have 
\[\dfrac{\max\{\phi_\sigma(nt-k): t\in [0,1]\setminus(\frac{k}{n}-\delta, \frac{k}{n}+\delta) \}}{\min\{\phi_\sigma(nt-k): t\in[\frac{k}{n}, \frac{k}{n}+\delta^2]\}}\to 0\quad \text{as}\,\, n\to \infty,\]
for $0<\delta<1.$ \\
\textbf{Case 2:\,$\frac{k}{n}>\frac{1}{2}.$} 
Repeating the first few steps of case 1, we get
\begin{eqnarray*}
\left|c_{n,k}-f\left(\frac{k}n{}\right)\right| 
&\leq& \epsilon + 2 M \dfrac{\displaystyle{\int_{[0,1]\setminus(\frac{k}{n}-\delta, \frac{k}{n}+\delta)}}\phi_\sigma(nt-k)\,d\rho(t)}{\int_{0}^1\phi_\sigma(nt-k)\,d\rho(t)} \\
&\leq& \epsilon+ 2M\,\dfrac{\max\{\phi_\sigma(nt-k): t\in [0,1] \setminus(\frac{k}{n}-\delta, \frac{k}{n}+\delta) \}}{\min\{\phi_\sigma(nt-k): t\in[\frac{k}{n}-\delta^2, \frac{k}{n}]\}} \ \ \dfrac{\rho(I)}{\rho([\frac{k}{n}-\delta^2, \frac{k}{n}])}
\end{eqnarray*}
By Lemma \ref{mmpos} and the assumption of the theorem, we get the desired result. Hence, the proof is completed.
\end{proof}
In the following lemma, we show the boundedness of NN operator $S_n^\rho$ for functions in $\Lp.$
\begin{lemma} \label{uoplpboubd}
Let $1\leq p<\infty$. Then for $f\in \Lp$, we have
$$\|S^{\rho}_nf\|_{\Lp}\leq\|f\|_{\Lp}.$$
\end{lemma}

\begin{proof} 
Let $f\in\Lp$. Then, we have
\begin{equation*}
    \|S^{\rho}_nf\|^p_{\Lp}=\int_{0}^{1}\left|\sum_{k=0}^{n}\dfrac{\int_{0}^1f(t)\phi_\sigma(nt-k)d\rho(t)}{\int_{0}^1\phi_\sigma(nt-k)d\rho(t)} \dfrac{\phi_\sigma(nx-k)}{\sum_{k=0}^{n}\phi_\sigma(nx-k)}\right|^p d\rho(x). 
    \end{equation*}
Using Jensen's inequality and Holder's inequality, we obtain
\begin{eqnarray*}
\|S^{\rho}_nf\|^p_{\Lp} &\leq& \displaystyle{\int_{0}^1} \left(\displaystyle{\sum_{k=0}^n}\left(\dfrac{\int_0^1f(t)\phi_\sigma(nt-k)d\rho(t)}{\int_0^1\phi_\sigma(nt-k)d\rho(t)}\right)^p\dfrac{\phi_\sigma(nx-k)}{\sum_{k=0}^n\phi_\sigma(nx-k)}\right)d\rho(x) \\
&\leq& \dfrac{1}{\phi_\sigma(1)} \displaystyle{\sum_{k=0}^n}\dfrac{\left(\int_0^1f(t)\phi_\sigma(nt-k)d\rho(t)\right)^p}{\left(\int_0^1\phi_\sigma(nt-k)d\rho(t)\right)^{p-1}} \\
&\leq& \dfrac{1}{\phi_\sigma(1)} \displaystyle{\sum_{k=0}^n}\dfrac{\left(\int_0^1|f(t)|^p\phi_\sigma(nt-k)d\rho(t)\right)\left(\int_0^1\phi_\sigma(nt-k)d\rho(t)\right)^{\frac{p}{q}}}{\left(\int_0^1\phi_\sigma(nt-k)d\rho(t)\right)^{p-1}} \\
&\leq& \dfrac{1}{\phi_\sigma(1)} \displaystyle{\sum_{k=0}^n}\int_0^1|f(t)|^p\phi_\sigma(nt-k)d\rho(t) \\
&\leq& \|f\|_{\Lp}^p.     
\end{eqnarray*} 
Thus, the proof is completed. 
\end{proof} 

Now, we prove the $\Lp$ convergence of $S^{\rho}_n.$
 \begin{corollary} \label{oplpcon}
    Let $1\leq p<\infty.$ Under the assumptions of Theorem \ref{uniapp}, for $f\in \Lp,$ we have
$$\lim_{n\to\infty}\|S^{\rho}_nf-f\|_{\Lp}=0.$$
\end{corollary}
\begin{proof}

From theorem \ref{uniapp}, it is easy to see that for $g\in C(I)$, we have \begin{equation}
    \lim_{n\to \infty}\|S^{\rho}_ng-g\|^p_{\Lp}=0.\label{contlp}
\end{equation}
Applying the triangle inequality and the lemma \ref{uoplpboubd}, we obtain 
   \begin{eqnarray*}
    \|S^{\rho}_nf-f\|_{\Lp} 
    &\leq&  \|S^{\rho}_nf-S^{\rho}_ng\|_{\Lp}+ \|S^{\rho}_ng-g\|_{\Lp} + \|f-g\|_{\Lp}\\
    &\leq& 2\|f-g\|_{\Lp} + \|S^{\rho}_ng-g\|_{\Lp}.
   \end{eqnarray*}
 Using the density of $C(I)$ in $\Lp,$ and (\ref{contlp}) we get the desired result.
\end{proof}
     
In the following theorem, we estimate the error in the approximation in terms of $\mathcal{K}-$functional.
\begin{theorem} \label{errorest}
    Let $1\leq p<\infty$. Suppose that $M_p(\phi_\sigma)<\infty.$ Then for $f\in \Lp$, we have
    $$\|S^{\rho}_nf-f\|_{\Lp}\leq C\,\mathcal{K}\left(f,\frac{1}{n}\right). $$
\end{theorem}
\begin{proof} We know that $S^{\rho}_n(1)=1$ and $\|S^{\rho}_n\|_{\Lp}=1$.
For any $g\in W^{1,\infty}(I)$, we get
\begin{align}
\|S^{\rho}_nf-f\|_{\Lp}
 &\leq \|S^{\rho}_nf-S^{\rho}_ng\|_{\Lp}+\|S^{\rho}_ng-g\|_{\Lp}+\|f-g\|_{\Lp} \nonumber \\
 &\leq 2 \|f-g\|_{\Lp}+\|S^{\rho}_ng-g\|_{\Lp}. \label{Kestimate}
\end{align}
Now we estimate $\|S^{\rho}_ng-g\|_{\Lp}$ for $g\in W^{1,\infty}(I,\rho)$. 
For all $g\in W^{1,\infty}(I,\rho)$, we have 
\begin{align}
|g(t)-g(x)|\leq \|g\|_{1,\infty}|t-x|, \ \ \forall \, x,\,t\in I.\label{gestimate}
\end{align}
Since $S^{\rho}_n$ is a positive linear operator and it reproduces the constant, so we get
 $$|S^{\rho}_ng(x)-g(x)|=|S^{\rho}_n(g(t)-g(x))(x)|\leq S^{\rho}_n(|g(t)-g(x)|)(x)\leq \|g\|_{1,\infty} S^{\rho}_n(|t-x|)(x),$$
 
 For $x \in I$, we have
\begin{align}
    \|S^{\rho}_ng - g\|_{\Lp} &\leq \|g\|_{1,\infty} \|S^{\rho}_n(|t-x|)\|_{\Lp}. \label{gnestimate}
\end{align}

Now we estimate $S^{\rho}_n(|t-x|)(x).$ We begin by writing it as:
\begin{align*}
S^{\rho}_n(|t-x|)(x) &= \sum_{k=0}^n \frac{\int_0^1 |t-x| \phi_\sigma(nt-k) d\rho(t)}{\int_0^1 \phi_\sigma(nt-k) d\rho(t)} \frac{\phi_\sigma(nx-k)}{\sum_{k=0}^n \phi_\sigma(nx-k)} \\
&\leq \sum_{k=0}^n \frac{\int_0^1 |t - \frac{k}{n}| \phi_\sigma(nt-k) d\rho(t)}{\int_0^1 \phi_\sigma(nt-k) d\rho(t)} \frac{\phi_\sigma(nx-k)}{\sum_{k=0}^n \phi_\sigma(nx-k)} \\
&\quad + \sum_{k=0}^n \frac{\int_0^1 \left|\frac{k}{n} - x\right| \phi_\sigma(nt-k) d\rho(t)}{\int_0^1 \phi_\sigma(nt-k) d\rho(t)} \frac{\phi_\sigma(nx-k)}{\sum_{k=0}^n \phi_\sigma(nx-k)} \\
&= I_1 + I_2.
\end{align*}

Taking the $\Lp$-norm on both sides of the above expression, we get
\begin{align}
    \|S^{\rho}_n(|t-x|)\|_{\Lp} &\leq \|I_1\|_{\Lp} + \|I_2\|_{\Lp}. \label{modestimate}
\end{align}

We first estimate $\|I_1\| _{\Lp}$. Using Jensen's inequality twice, we obtain 
 \begin{align}
\|I_1\|^p_{\Lp} &= \displaystyle{\int_0^1\left(\sum_{k=0}^n\dfrac{\int_0^1|t-\frac{k}{n}|\,|\phi_\sigma(nt-k)|d\rho(t)}{\int_0^1\phi_\sigma(nt-k)d\rho(t)}\dfrac{\phi_\sigma(nx-k)}{\sum_{k=0}^n\phi_\sigma(nx-k)}\right)^p}d\rho(x) \nonumber\\
&\leq \displaystyle{\int_0^1\sum_{k=0}^n\left(\dfrac{\int_0^1|t-\frac{k}{n}|\,|\phi_\sigma(nt-k)|d\rho(t)}{\int_0^1\phi_\sigma(nt-k)d\rho(t)}\right)^p\dfrac{\phi_\sigma(nx-k)}{\sum_{k=0}^n\phi_\sigma(nx-k)}}d\rho(x) \nonumber \\
&\leq \displaystyle{\int_0^1\sum_{k=0}^n\dfrac{\int_0^1|t-\frac{k}{n}|^p\,|\phi_\sigma(nt-k)|d\rho(t)}{\left(\int_0^1\phi_\sigma(nt-k)d\rho(t)\right)}\dfrac{\phi_\sigma(nx-k)}{\sum_{k=0}^n\phi_\sigma(nx-k)}}d\rho(x) \quad \nonumber \\
&\leq \frac{1}{n^p\,\phi_\sigma(1)} \displaystyle{\sum_{k=0}^n}\int_0^1 |nt-k|^p\phi_\sigma(nt-k)d\rho(t) \nonumber\\
&\leq \frac{1}{n^p\,\phi_\sigma(1)} M_p(\phi_\sigma) \, \rho(I) 
=  \frac{C}{n^p}. \label{I1estimate}
\end{align}
Similarly, we estimate $\|I_2\|_{\Lp}.$ Again, using Jensen's inequality, we get 
\begin{align}
\|I_2\|^p_{\Lp} &= \displaystyle{\int_0^1 \left( \sum_{k=0}^n\left|\frac{k}{n}-x\right| \dfrac{\phi_\sigma(nx-k)}{\sum_{k=0} ^n\phi_\sigma(nx-k)} \right)^p} d\rho(x) \nonumber \\
&\leq \displaystyle{\int_0^1 \left( \sum_{k=0}^n\left|\frac{k}{n}-x\right|^p \dfrac{\phi_\sigma(nx-k)}{\sum_{k=0} ^n\phi_\sigma(nx-k)} \right)} d\rho(x) \nonumber \\
&\leq \frac{1}{\phi_\sigma(1)} \int_{0}^1 \left(\sum_{k=0}^n\left|\frac{k}{n}-x\right|^p \phi_\sigma(nx-k)\right)  d\rho(x) \nonumber \\
&\leq \dfrac{1}{n^p\,\phi_\sigma(1)} M_p(\phi_\sigma)\, \rho(I)
 = \frac{C}{n^p}. \label{I2estimate}
\end{align}
On combining the estimates (\ref{gnestimate})-(\ref{I2estimate}), we obtain
\begin{equation}
    \|S^{\rho}_ng-g\|_{\Lp}  \leq C \, \|g\|_{1,\infty} \frac{1}{n}. \label{npestimate}
\end{equation}
Substituting (\ref{npestimate}) into (\ref{Kestimate}), and taking the infimum over all $g\in W^{1,\infty}(I,\rho)$, we get the desired result.
\end{proof}

\section{Multivariate Neural Network operators with respect to arbitrary measures} \label{multi}
In this section, we analyze the approximation properties of multivariate neural network operators with respect to arbitrary measures. In particular, we derive the uniform approximation and the error bounds in terms of $\mathcal{K}$-functionals. Before proving these results, we recall the following multivariate neural network operator.
\begin{definition}
Let $f:I^d\to\R$ be a bounded function and $n\in\N$. Then the multivariate neural network operator $F_n$ is defined as follows (see \cite{Costar2}):
$$F_nf(x):=\displaystyle{\dfrac{\sum_{k_1=0}^n\sum_{k_2=0}^n\ldots\sum_{k_d=0}^n f\left(\frac{\beta}{n}\right)\,\Phi_\sigma(nx_1-k_1,nx_2-k_2,\ldots,nx_d-k_d)}{\sum_{k_1=0}^n\sum_{k_2=0}^n\ldots\sum_{k_d=0}^n\Phi_\sigma(nx_1-k_1,nx_2-k_2,\ldots,nx_d-k_d)}},\quad x\in I^d$$
 where $\beta$ is a multi-index such that $\beta=(k_1,k_2,\ldots,k_d),$ and $\frac{\beta}{n}=\left(\frac{k_1}{n}, \frac{k_2}{n},\ldots,\frac{k_d}{n}\right)$. \end{definition}
First, we recall the following theorem from \cite{Costar2}.
\begin{theorem} \label{muni}
    Let $f:I^d\to\R$ be a bounded. If $f$ is continuous at $\X$, then $$\lim_{n\to\infty}
    F_nf(\X)=f(\X).$$ Further, if $f\in C(I^d)$, then we have
    $$\lim_{n\to\infty}\|F_nf-f\|_{\infty}=0.$$
\end{theorem}

 Now, we prove the uniform convergence of the NN operator $S_n^\rho$ for the functions in $C(I^d)$.
\begin{theorem} \label{munifo}
Let $\rho$ be a strictly positive bounded Borel measure on $I^d$. Suppose for $0<\delta<1,$ either
\begin{align}\lim_{n\rightarrow \infty} \dfrac{\max\{\phi_\sigma(nt-k): t\in [0,1]\setminus (\frac{k}{n}-\delta, \frac{k}{n}+\delta) \}}{\min\{\phi_\sigma(nt-k): t\in[\frac{k}{n}, \frac{k}{n}+\delta^2]\}}=0  \end{align} or 
\begin{align}\lim_{n\rightarrow \infty} \dfrac{\max\{\phi_\sigma(nt-k): t\in  [0,1]\setminus (\frac{k}{n}-\delta, \frac{k}{n}+\delta)\}}{\min\{\phi_\sigma(nt-k): t\in[\frac{k}{n}-\delta^2, \frac{k}{n}]\}}=0. \end{align}
 Then for $f\in C(I^d)$, we have $$\lim_{n\to\infty}\|S^{\rho}_nf-f\|_{\infty}=0.$$
\end{theorem}

\begin{proof}
    For all $x\in I^d$, we have
\begin{eqnarray*}
|S^{\rho}_nf(\X)-f(\X)|
&\leq& |S^{\rho}_nf(\X)-F_nf(\X| + |F_nf(\X)-f(\X)|\\
&:=& I_1 + I_2.
\end{eqnarray*}
Using Theorem \ref{muni}, we have $I_2\to 0$ uniformly as $n\to\infty$. We now estimate the term $I_1.$
 \begin{eqnarray*}
     I_1 &\leq& \left|\displaystyle{\dfrac{\sum_{k_1=0}^n\sum_{k_2=0}^n\ldots\sum_{k_d=0}^n\left|c_{n,\beta}-f\left(\frac{\beta}{n}\right)\right|\,\Phi_\sigma(nx_1-k_1,nx_2-k_2,\ldots,nx_d-k_d)}{\sum_{k_1=0}^n\sum_{k_2=0}^n\ldots\sum_{k_d=0}^n\Phi_\sigma(nx_1-k_1,nx_2-k_2,\ldots,nx_d-k_d)}}\right| \\
     &\leq& \max_{\frac{\beta}{n}\in I^d} \left|c_{n,\beta}-f\left(\frac{\beta}{n}\right)\right|.
 \end{eqnarray*}
 So, it is enough to show that $\left|c_{n,\beta}-f\left(\frac{\beta}{n}\right)\right|\to0$ as $n\to\infty.$ Suppose $\frac{\beta}{n}\in I_A.$
 Since $f$ is uniformly continuous in $I^d,$ for every $\epsilon>0$, there exists $0<\delta<\frac{1}{2}$ such that $|f(x)-f(y)|<\epsilon$ whenever $\displaystyle \X,\,\mathbf{y}\in U_{\delta}(\beta)\cap I^d.$ Using the boundedness and uniform continuity of $f$, we have \\

 \noindent 
 $\displaystyle\left|c_{n,\beta}-f\left(\frac{\beta}{n}\right)\right| $
\begin{align*}
& \leq \dfrac{\int_{I^d}\left|f(\mathbf{t})-f\left(\frac{\beta}{n}\right)\right|\,\Phi_\sigma(nt_1-k_1, nt_2-k_2, \ldots, nt_d-k_d)\,d\rho(t)}{\int_{I^d} \Phi_\sigma(nt_1-k_1, nt_2-k_2, \ldots, nt_d-k_d)\,d\rho(t)} \\
& = \dfrac{\int_{U_{\delta}\left(\frac{\beta}{n}\right) \cap I^d} \left| f(\mathbf{t}) - f\left(\frac{\beta}{n}\right) \right|\,\Phi_\sigma(nt_1-k_1, nt_2-k_2, \ldots, nt_d-k_d)\,d\rho(t)}{\int_{I^d} \Phi_\sigma(nt_1-k_1, nt_2-k_2, \ldots, nt_d-k_d)\,d\rho(t)} \\
& \quad + \dfrac{\int_{I^d \setminus U_{\delta}\left(\frac{\beta}{n}\right)} \left| f(\mathbf{t}) - f\left(\frac{\beta}{n}\right) \right|\, \Phi_\sigma(nt_1-k_1, nt_2-k_2, \ldots, nt_d-k_d)\,d\rho(t)}{\int_{I^d} \Phi_\sigma(nt_1-k_1, nt_2-k_2, \ldots, nt_d-k_d)\,d\rho(t)} \\
& \leq \epsilon + 2M \, \dfrac{\int_{I^d \setminus U_{\delta}\left(\frac{\beta}{n}\right)} \Phi_\sigma(nt_1-k_1, nt_2-k_2, \ldots, nt_d-k_d)\,d\rho(t)}{\int_{V_{A,\delta^2}\left(\frac{\beta}{n}\right)} \Phi_\sigma(nt_1-k_1, nt_2-k_2, \ldots, nt_d-k_d)\,d\rho(t)} \\
& \leq \epsilon + 2M \, \dfrac{\max\left\{ \Phi_\sigma(nt_1-k_1, nt_2-k_2, \ldots, nt_d-k_d) : (t_1, t_2, \ldots, t_d) \in I^d \setminus U_{\delta}\left(\frac{\beta}{n}\right) \right\}}{\min\left\{ \Phi_\sigma(nt_1-k_1, nt_2-k_2, \ldots, nt_d-k_d) : (t_1, t_2, \ldots, t_d) \in V_{A, \delta^2}\left(\frac{\beta}{n}\right) \right\}} \\
& \quad \times \dfrac{\rho(I^d)}{\rho\left( V_{A, \delta^2}\left(\frac{\beta}{n}\right) \right)}.
\end{align*}
Since $(t_1,t_2,\ldots, t_d)\in V_{A, \delta^2}\left(\frac{\beta}{n}\right),$ then $t_i\in [\frac{k_i}{n}-\delta^2, \frac{k_i}{n}]\,\text{or}\, [\frac{k_i}{n}, \frac{k_i}{n}+\delta^2].$ By the assumption of the theorem and noting that $\phi_\sigma$ is a non decreasing for $x<0$ and non increasing for $x\geq0$ and $\Phi_\sigma(nt_1-k_1,nt_2-k_2,\ldots,nt_d-k_d)=\prod_{i=1}^d\phi_\sigma(nt_i-k_i)$, we have 
    $$\lim_{n\to\infty}  \dfrac{ \max\{\Phi_\sigma(nt_1-k_1,nt_2-k_2,\ldots,nt_d-k_d):(t_1, t_2,\ldots, t_d)\in I^d\setminus U_{\delta}\left(\frac{\beta}{n}\right)\}}{\min\{\Phi_\sigma(nt_1-k_1,nt_2-k_2,\ldots,nt_d-k_d):(t_1, t_2,\ldots, t_d)\in V_{A,\delta^2}\left(\frac{\beta}{n}\right)\}}=0.$$
 This completes the proof. 
\end{proof}

In the following lemma, we show the boundedness of the NN operator $S_n^\rho$ for functions in $\LP.$
\begin{lemma} \label{oplpboubd}
Let $1\leq p<\infty$. Then for $f\in \LP$, we have
$$\|S^{\rho}_nf\|_{\LP}\leq\|f\|_{\LP}.$$
\end{lemma}

\begin{proof} 
Let $f\in \LP$. Then, using Jensen's and Holder's inequality, we get
\begin{align*}
    \|S_n^{\rho}f\|_{\LP}^p  &= \int_{I^d} \left| \sum_{k_1=0}^n \ldots \sum_{k_d=0}^n \left( \frac{\int_{I^d} f(t) \prod_{j=1}^d \phi_\sigma(nt_j - k_j) \, d\rho(t)}{\int_{I^d} \prod_{j=1}^d \phi_\sigma(nt_j - k_j) \, d\rho(t)} \right) \right. \\
&\quad \left. \times \frac{\prod_{j=1}^d \phi_\sigma(nx_j - k_j)}{\sum_{k_1=0}^n \ldots \sum_{k_d=0}^n \prod_{j=1}^d \phi_\sigma(nx_j - k_j)} \right|^p d\rho(x) \\
&\leq \int_{I^d}  \sum_{k_1=0}^n \ldots \sum_{k_d=0}^n \left| \frac{\int_{I^d} f(t) \prod_{j=1}^d \phi_\sigma(nt_j - k_j) \, d\rho(t)}{\int_{I^d} \prod_{j=1}^d \phi_\sigma(nt_j - k_j) \, d\rho(t)} \right|^p  \\
&\quad  \times \frac{\prod_{j=1}^d \phi_\sigma(nx_j - k_j)}{\sum_{k_1=0}^n \ldots \sum_{k_d=0}^n \prod_{j=1}^d \phi_\sigma(nx_j - k_j)}  d\rho(x) \\
 &\leq \frac{1}{(\phi_\sigma(1))^d}  \sum_{k_1=0}^n \ldots \sum_{k_d=0}^n  \frac{\left(\int_{I^d} f(t) \prod_{j=1}^d \phi_\sigma(nt_j - k_j) \, d\rho(t)\right)^p}{\left(\int_{I^d} \prod_{j=1}^d \phi_\sigma(nt_j - k_j) \, d\rho(t)\right)^{p-1}}\\
 &\leq \frac{1}{(\phi_\sigma(1))^d}  \sum_{k_1=0}^n \ldots \sum_{k_d=0}^n  \frac{\left(\int_{I^d} |f(t)|^p \prod_{j=1}^d \phi_\sigma(nt_j - k_j) \, d\rho(t)\right)  }{\left(\int_{I^d} \prod_{j=1}^d \phi_\sigma(nt_j - k_j) \, d\rho(t)\right)^{p-1}}\\
 &\quad\times \left(\int_{I^d} \prod_{j=1}^d \phi_\sigma(nt_j - k_j) \, d\rho(t)\right)^{\frac{p}{q}} \\
 &\leq \frac{1}{(\phi_\sigma(1))^d}  \sum_{k_1=0}^n \ldots \sum_{k_d=0}^n  \left(\int_{I^d} |f(t)|^p \prod_{j=1}^d \phi_\sigma(nt_j - k_j) \, d\rho(t)\right) \\
 &\leq \|f\|_{\LP}.
\end{align*}
Hence, the proof is completed. 
\end{proof}
Next, we obtain the error bounds in terms of the $\mathcal{K}$-functionals.
\begin{theorem} 
    Let $1\leq p<\infty$. Suppose that $M_p(\phi_\sigma)<\infty.$ Then for $f\in \LP$, we have
    $$\|S^{\rho}_nf-f\|_{\LP}\leq C\,\mathcal{K}\left(f,\frac{1}{n}\right). $$
\end{theorem}
\begin{proof} 
For any $g\in W^{1,\infty}(I^d,\rho)$, we get
\begin{align}
\|S^{\rho}_nf-f\|_{\LP} &\leq \|S^{\rho}_nf-S^{\rho}_ng\|_{\LP}+\|S^{\rho}_ng-g\|_{\LP}+\|f-g\|_{\LP} \nonumber \\
 &\leq 2 \|f-g\|_{\LP}+\|S^{\rho}_ng-g\|_{\LP}. \label{mKesti}
\end{align}
For all $g\in W^{1,\infty}(I^d,\rho)$, we have \begin{align}
|g(\mathbf{t})-g(\mathbf{x})|\leq \|g\|_{1,\infty}\sum_{i=1}^d|t_i-x_i|, \,\, \forall \, \X,\,\mathbf{t}\in  I^d.\end{align}
Since $S^{\rho}_n$ is a positive linear operator and reproduces the constant functions, we get
$$|S^{\rho}_ng(\mathbf{x})-g(\mathbf{x})|=|S^{\rho}_n(g(\mathbf{t})-g(\mathbf{x}))(\mathbf{x})|\leq S^{\rho}_n(|g(\mathbf{t})-g(\mathbf{x})|)(\mathbf{x})\leq \|g\|_{1,\infty}\sum_{i=1}^d S^{\rho}_n(|t_i-x_i|)(\mathbf{x}).$$
 Taking the $L^p$ norm on both sides, we get
 \begin{align}
     \|S^{\rho}_ng-g\|_{\LP}  \leq \|g\|_{1,\infty} \sum_{i=1}^d \|S^{\rho}_n(|\pi_i(\mathbf{t})-\pi_i(\X)|)\|_{\LP}, \label{mgesti}
 \end{align}
where $\pi_i:I^d\to\R$ is the projection on the $i^{th}$ coordinate. Now, we estimate  $S^{\rho}_n(|\pi_i(\mathbf{t})-\pi_i(\X)|).$
Let $i=1,2,\ldots,d.$ Then, we have
\begin{align*}
S^{\rho}_n(|\pi_i(\mathbf{t})-\pi_i(\X)|) &=   \sum_{k_1=0}^n \ldots \sum_{k_d=0}^n \left( \frac{\int_{I^d} |\pi_i(\mathbf{t})-\pi_i(\X)| \prod_{j=1}^d \phi_\sigma(nt_j - k_j) \, d\rho(t)}{\int_{I^d} \prod_{j=1}^d \phi_\sigma(nt_j - k_j) \, d\rho(t)} \right) \\
&\quad  \times \frac{\prod_{j=1}^d \phi_\sigma(nx_j - k_j)}{\sum_{k_1=0}^n \ldots \sum_{k_d=0}^n \prod_{j=1}^d \phi_\sigma(nx_j - k_j)}   \\
&\leq   \sum_{k_1=0}^n \ldots \sum_{k_d=0}^n \left( \frac{\int_{I^d} |\pi_i(\mathbf{t}) - \frac{k_i}{n}| \prod_{j=1}^d \phi_\sigma(nt_j - k_j) \, d\rho(t)}{\int_{I^d} \prod_{j=1}^d \phi_\sigma(nt_j - k_j) \, d\rho(t)} \right)  \\
&\quad  \times \frac{\prod_{j=1}^d \phi_\sigma(nx_j - k_j)}{\sum_{k_1=0}^n \ldots \sum_{k_d=0}^n \prod_{j=1}^d \phi_\sigma(nx_j - k_j)}  \\
&+   \sum_{k_1=0}^n \ldots \sum_{k_d=0}^n \left( \frac{\int_{I^d} \left| \frac{k_i}{n} - \pi_i(\X) \right| \prod_{j=1}^d \phi_\sigma(nt_j - k_j) \, d\rho(t)}{\int_{I^d} \prod_{j=1}^d \phi_\sigma(nt_j - k_j) \, d\rho(t)} \right) \\
&\quad  \times \frac{\prod_{j=1}^d \phi_\sigma(nx_j - k_j)}{\sum_{k_1=0}^n \ldots \sum_{k_d=0}^n \prod_{j=1}^d \phi_\sigma(nx_j - k_j)} \\
&\leq \sum_{k_1=0}^n \ldots \sum_{k_d=0}^n \left( \frac{\int_{I^d} |\pi_i(\mathbf{t}) - \frac{k_i}{n}| \prod_{j=1}^d \phi_\sigma(nt_j - k_j) \, d\rho(t)}{\int_{I^d} \prod_{j=1}^d \phi_\sigma(nt_j - k_j) \, d\rho(t)} \right)  \\
&\quad  \times \frac{\prod_{j=1}^d \phi_\sigma(nx_j - k_j)}{\sum_{k_1=0}^n \ldots \sum_{k_d=0}^n \prod_{j=1}^d \phi_\sigma(nx_j - k_j)}  \\
&+ \sum_{k_1=0}^n \ldots \sum_{k_d=0}^n \frac{\left( \int_{I^d} \left| \frac{k_i}{n} - \pi_i(\X) \right| \prod_{j=1}^d \phi_\sigma(nt_j - k_j) \, d\rho(t) \right)}{\sum_{k_1=0}^n \ldots \sum_{k_d=0}^n \prod_{j=1}^d \phi_\sigma(nx_j - k_j)} \\
&=: I_1 + I_2.
\end{align*} 
Taking the $L^p$ norm on both sides of the expression, we have 
\begin{align} 
\|S^{\rho}_n(|\pi_i(\mathbf{t})-\pi_i(\X)|)\|_{\LP} 
&\leq \|I_1\|_{\LP} + \|I_2\|_{\LP}. \label{msumesti} 
\end{align} 

Now we estimate $I_1$ and $I_2.$ Using Jensen's inequality twice, we obtain the following:
\begin{align}
\|I_2\|^p_{\LP} 
&= \int_{I^d} \left( \sum_{k_1=0}^n \ldots \sum_{k_d=0}^n \left| \frac{k_i}{n} - \pi_i(\X) \right| \frac{\prod_{j=1}^d \phi_\sigma(nx_j - k_j)} {\sum_{k_1=0}^n \ldots \sum_{k_d=0}^n \prod_{j=1}^d \phi_\sigma(nx_j - k_j)} \right)^p d\rho(x) \nonumber \\
&\leq \int_{I^d} \sum_{k_1=0}^n \ldots \sum_{k_d=0}^n \left| \frac{k_i}{n} - \pi_i(\X) \right|^p \frac{\prod_{j=1}^d \phi_\sigma(nx_j - k_j)} {\sum_{k_1=0}^n \ldots \sum_{k_d=0}^n \prod_{j=1}^d \phi_\sigma(nx_j - k_j)} d\rho(x) \nonumber \\
&\leq \frac{1}{(\phi_\sigma(1))^d} \int_{I^d} \sum_{k_1=0}^n \ldots \sum_{k_d=0}^n \left| \frac{k_i}{n} - \pi_i(\X) \right|^p \prod_{j=1}^d \phi_\sigma(nx_j - k_j) d\rho(x) \nonumber \\
&\leq \frac{1}{(\phi_\sigma(1))^d} \int_{I^d} \left( \sum_{k_1=0}^n \ldots \sum_{k_d=0}^n \left| \frac{k_i}{n} - \pi_i(\X) \right|^p \prod_{j=1}^d \phi_\sigma(nx_j - k_j) \right) d\rho(x) \nonumber \\
&=  \int_{I^d} \left( \sum_{k_1=0}^n \ldots \sum_{k_d=0}^n \left( \sum_{k_i=0}^n \left| \frac{k_i}{n} - \pi_i(\X) \right|^p \phi_\sigma(nx_i - k_1) \right) \prod_{\substack{j=1 \\ j\neq i}}^d \phi_\sigma(nx_j - k_j)\right) d\rho(x) \nonumber \\
&= \frac{1}{(\phi_\sigma(1))^p} \int_{I^d} \left( \sum_{k_i=0}^n \left| \frac{k_i}{n} - \pi_i(\X) \right|^p \phi_\sigma(nx_i - k_i) \right) \nonumber \\
&\qquad \times \left( \sum_{k_1=0}^n \ldots \sum_{k_d=0}^n \prod_{j=1,\, j \neq i}^d \phi_\sigma(nx_j - k_j) \right) d\rho(x) \nonumber \\
&\leq \frac{1}{(\phi_\sigma(1))^d} \frac{1}{n^p} M_p(\phi_\sigma) \int_{I^d} \left( \prod_{j=1,\, j \neq i}^d \sum_{k_j=0}^n \phi_\sigma(nx_j - k_j) \right) d\rho(x) \nonumber \\
&\leq \frac{1}{(\phi_\sigma(1))^d} M_p(\phi_\sigma) \frac{\rho(I^d)}{n^p} = \frac{C}{n^p}. \label{mI2esti}
\end{align}

Similarly, we estimate $\|I_1\|_{\LP}$ as follows. Again using Jensen's inequality, we obtain
\begin{align}
\|I_1\|^p_{\LP} 
&\leq \int_{I^d} \left| \sum_{k_1=0}^n \ldots \sum_{k_d=0}^n \left( \frac{\int_{I^d} \left|\pi_i(\mathbf{t}) - \frac{k_i}{n}\right| \prod_{j=1}^d \phi_\sigma(nt_j - k_j) \, d\rho(t)}{\int_{I^d} \prod_{j=1}^d \phi_\sigma(nt_j - k_j) \, d\rho(t)} \right) \right. \nonumber \\
&\quad \times \left. \frac{\prod_{j=1}^d \phi_\sigma(nx_j - k_j)}{\sum_{k_1=0}^n \ldots \sum_{k_d=0}^n \prod_{j=1}^d \phi_\sigma(nx_j - k_j)} \right|^p d\rho(x)\nonumber \\
&\leq \int_{I^d} \sum_{k_1=0}^n \ldots \sum_{k_d=0}^n \left( \frac{\int_{I^d} \left|\pi_i(\mathbf{t}) - \frac{k_i}{n}\right| \prod_{j=1}^d \phi_\sigma(nt_j - k_j) \, d\rho(t)}{\int_{I^d} \prod_{j=1}^d \phi_\sigma(nt_j - k_j) \, d\rho(t)} \right)^p \nonumber\\
&\quad \times \frac{\prod_{j=1}^d \phi_\sigma(nx_j - k_j)}{\sum_{k_1=0}^n \ldots \sum_{k_d=0}^n \prod_{j=1}^d \phi_\sigma(nx_j - k_j)}  d\rho(x) \nonumber
\end{align}

\begin{align}
&\leq \int_{I^d}\sum_{k_1=0}^n \ldots \sum_{k_d=0}^n \left( \frac{\int_{I^d} \left|\pi_i(\mathbf{t}) - \frac{k_i}{n}\right|^p \prod_{j=1}^d \phi_\sigma(nt_j - k_j) \, d\rho(t)}{\int_{I^d} \prod_{j=1}^d \phi_\sigma(nt_j - k_j) \, d\rho(t)} \right) \nonumber \\
&\quad \times \frac{\prod_{j=1}^d \phi_\sigma(nx_j - k_j)}{\sum_{k_1=0}^n \ldots \sum_{k_d=0}^n \prod_{j=1}^d \phi_\sigma(nx_j - k_j)}   d\rho(x) \nonumber \\
&\leq \sum_{k_1=0}^n \ldots \sum_{k_d=0}^n  \left( \frac{\int_{I^d} \left|\pi_i(\mathbf{t}) - \frac{k_i}{n}\right|^p \prod_{j=1}^d \phi_\sigma(nt_j - k_j) \, d\rho(t)}{\sum_{k_1=0}^n \ldots \sum_{k_d=0}^n \prod_{j=1}^d \phi_\sigma(nx_j - k_j)} \right) \nonumber\\
&\leq \frac{1}{(\phi_\sigma(1))^d} \sum_{k_1=0}^n \ldots \sum_{k_d=0}^n  \left( \int_{I^d} \left|\pi_i(\mathbf{t}) - \frac{k_i}{n}\right|^p \prod_{j=1}^d \phi_\sigma(nt_j - k_j)\, d\rho(t)\right) \nonumber\\
&\leq \frac{1}{(\phi_\sigma(1))^d} \int_{I^d} \sum_{k_1=0}^n \ldots \sum_{k_d=0}^n  \left( \sum_{k_i=0}^n  \left|\pi_i(\mathbf{t}) - \frac{k_i}{n}\right|^p \phi_\sigma(nx_i-k_i)\right)\nonumber\\
&\quad \times \left( \prod_{j=1, j\neq i }^d \phi_\sigma(nt_j - k_j) \right) d\rho(t) \nonumber \\
&\leq \frac{M_p(\phi_\sigma)}{n^p (\phi_\sigma(1))^d} \int_{I^d} \sum_{k_1=0}^n \ldots \sum_{k_d=0}^n \left( \prod_{j=1, j\neq i }^d \phi_\sigma(nt_j - k_j) \right)d\rho(t) \nonumber \\
&=\frac{M_p(\phi_\sigma)}{(\phi_\sigma(1))^d} \frac{1}{n^p}\int_{I^d} \left( \prod_{j=1,\, j\neq i}^d \sum_{k_j=0}^n \phi_\sigma(nt_j-k_j)\right) d\rho(t) \nonumber \\
&\leq \frac{M_p(\phi_\sigma)}{(\phi_\sigma(1))^d} \frac{\rho(I^d)}{n^p} = \frac{C}{n^p}. \label{mI1esti}
\end{align}
Combining (\ref{mgesti})-(\ref{mI1esti}), we obtain \begin{align}
     \|S^{\rho}_ng-g\|_{\LP}  \leq C\,d\,\|g\|_{1,\infty} \frac{1}{n}. \label{mgfinalesti}
\end{align}
Using (\ref{mgfinalesti}) in (\ref{mKesti}), and taking the infimum over $g\in W^{1,\infty}(I^d,\rho)$, we get the desired result.
\end{proof}

Now, we prove the $\LP$ convergence of $S^{\rho}_n.$
 \begin{corollary} \label{moplpcon}
    Let $1\leq p<\infty.$ For $f\in \LP,$ we have
$$\lim_{n\to\infty}\|S^{\rho}_nf-f\|_{\LP}=0.$$
\end{corollary}
\begin{proof} By Theorem \ref{munifo}, it is easy to see that for all $g\in C(I^d)$, we have 
\begin{equation} \lim_{n\to\infty}\|S^{\rho}_ng-g\|_{\LP}=0.  \label{moplpconti}\end{equation}
Applying the triangle inequality, and Lemma \ref{oplpboubd}, we obtain 
   \begin{eqnarray*} 
    \|S^{\rho}_nf-f\|_{\LP} 
    &\leq&  \|S^{\rho}_nf-S^{\rho}_ng\|_{\LP}+ \|S^{\rho}_ng-g\|_{\LP} + \|f-g\|_{\LP}\\
    &\leq& 2\|f-g\|_{\LP} + \|S^{\rho}_ng-g\|_{\LP}.
   \end{eqnarray*}
   Since $C(I^d)$ is dense $\LP,$ so for $f\in \LP$, there exists a function $g\in C(I^d)$ such that 
  \begin{equation} \|f-g\|_{\LP}<\epsilon. \label{mcontdense}\end{equation}
 On combining (\ref{moplpconti})-(\ref{mcontdense}), we get the desired result.
\end{proof}

\section{Examples of activation functions and Implementation Results}
\subsection{Examples of activation functions}
\label{example}
In this section, we take some particular activation functions and verify the assumption of the theorems for one and multidimensional NN operators. As a first example, we consider the following logistic function: $$\sigma(x)=\dfrac{1}{1+e^{-x}},\, x\in \R.$$

\begin{example}
    
\label{elogis}
    Let $\sigma(x)=\dfrac{1}{1+e^{-x}}.$ Using (\ref{kernel}), we can write  
    $$\phi_\sigma(x)=\dfrac{1}{2}\left(\dfrac{1}{1+e^{-(x+1)}}-\dfrac{1}{1+e^{-(x-1)}}\right).$$ 
    
    Given that $\phi_\sigma(nt-k)$ is a positive function that increases until $t=\dfrac{k}{n}$, and then decreases symmetrically about the point $t=\dfrac{k}{n}$, we can assume, without loss of generality,  that $\frac{k}{n}\leq\frac{1}{2}.$ Hence, we have the following: 
\[\dfrac{\max\{\phi_\sigma(nt-k): t\in [0,1] \setminus(\frac{k}{n}-\delta, \frac{k}{n}+\delta) \}}{\min\{\phi_\sigma(nt-k): t\in[\frac{k}{n}, \frac{k}{n}+\delta^2]\}}= \dfrac{\dfrac{1}{1+e^{-(n\delta+1)}}-\dfrac{1}{1+e^{-(n\delta-1)}}} {\dfrac{1}{1+e^{-(n\delta^2+1)}}-\dfrac{1}{1+e^{-(n\delta^2-1)}}}.\]
 Simplifying the RHS of the above expression, we obtain
\begin{eqnarray*}
\text{RHS}&=:& \dfrac{\dfrac{e^{-(n\delta-1)}-e^{-(n\delta+1)}}{1+e^{-(2n\delta)}+e^{-(n\delta-1)}+e^{-(n\delta+1)}}}{\dfrac{e^{-(n\delta^2-1)}-e^{-(n\delta^2+1)}}{1+e^{-(2n\delta^2)}+e^{-(n\delta^2-1)}+e^{-(n\delta^2+1)}}} \\
&=& \left(\dfrac{e^{-(n\delta-1)}-e^{-(n\delta+1)}}{e^{-(n\delta^2-1)}-e^{-(n\delta^2+1)}}\right) \times  \left(\dfrac{1+e^{-2n\delta^2}+e^{-(n\delta^2-1)}+e^{-(n\delta^2+1)}}{1+e^{-2n\delta}+e^{-(n\delta-1)}+e^{-(n\delta+1)}}\right) \\
&=& \dfrac{e^{-(n\delta)}}{e^{-(n\delta^2)}}\times \left(\dfrac{1+e^{-2n\delta^2}+e^{-(n\delta^2-1)}+e^{-(n\delta^2+1)}}{1+e^{-2n\delta}+e^{-(n\delta-1)}+e^{-(n\delta+1)}}\right) \\
&=& e^{n(\delta^2-\delta)} \times \left(\dfrac{1+e^{-2n\delta^2}+e^{-(n\delta^2-1)}+e^{-(n\delta^2+1)}}{1+e^{-2n\delta}+e^{-(n\delta-1)}+e^{-(n\delta+1)}}\right).
\end{eqnarray*}
Since $0<\delta<1$, we have $\delta^2-\delta<0.$ Therefore, we get 
\[
    \lim_{n\to\infty}\dfrac{\max\{\phi_\sigma(nt-k): t\in  [0,1]\setminus (\frac{k}{n}-\delta, \frac{k}{n}+\delta) \}}{\min\{\phi_\sigma(nt-k): t\in[\frac{k}{n}, \frac{k}{n}+\delta^2]\}}=0,\] where $0<\delta<1.$ This verifies the condition of the Theorem \ref{uniapp}. Further, it is easy to see that $M_p(\phi_\sigma)$ is finite for $1\leq p<\infty$ and hence conditions on Theorem \ref{errorest} is verified.
\end{example}

As a second example, we consider the tangent hyperbolic function 
$\sigma(x)=\dfrac{e^x - e^{-x}}{e^x + e^{-x}}.$ 

\begin{example} \label{etanh}
Let $\sigma(x)=\dfrac{e^x - e^{-x}}{e^x + e^{-x}}.$ Again, using (\ref{kernel}), we can write $$\phi_\sigma(x)=\dfrac{1}{2}\left(\dfrac{e^{x+1}-e^{-(x+1)}}{e^{(x+1)}+e^{-(x+1)}}-\dfrac{e^{(x-1)}-e^{-((x-1))}}{e^{(x-1)}+e^{-(x-1)}}\right).$$ 
We note that $\phi_\sigma(nt-k)$ is a positive function that increases until the point $\dfrac{k}{n}$, and decreases symmetrically about the point $\dfrac{k}{n}.$ We can assume, without loss of generality,  that $\frac{k}{n}\leq\frac{1}{2}.$ Thus, we have
\begin{eqnarray*}
 \dfrac{\max\{\phi_\sigma(nt-k): t\in [0,1] \setminus(\frac{k}{n}-\delta, \frac{k}{n}+\delta) \}}{\min\{\phi_\sigma(nt-k): t\in[\frac{k}{n}, \frac{k}{n}+\delta^2]\}} = \dfrac{\dfrac{e^{n\delta+1}-e^{-(n\delta+1)}}{e^{n\delta+1}+e^{-(n\delta+1)}}-\dfrac{e^{n\delta-1}-e^{-(n\delta-1)}}{e^{n\delta-1}+e^{-(n\delta-1)}}} {\dfrac{e^{n\delta^2+1}-e^{-(n\delta^2+1)}}{e^{n\delta^2+1}+e^{-(n\delta^2+1)}}-\dfrac{e^{n\delta^2-1}-e^{-(n\delta^2-1)}}{e^{n\delta^2-1}+e^{-(n\delta^2-1)}}}
 \end{eqnarray*}
Simplifying the above expression, we have 
\begin{eqnarray*}
\text{RHS} &=& \dfrac{e^{2n\delta} + e^2 - e^{-2} - e^{-2n\delta} - e^{-2n\delta} - e^{-2} + e^2 + e^{-2n\delta}}{e^{2n\delta} + e^2 + e^{-2} + e^{-2n\delta}} \\
& \times & \frac{e^{2n\delta^2} + e^2 + e^{-2} + e^{-2n\delta^2}}{e^{2n\delta^2} + e^2 - e^{-2} - e^{-2n\delta^2} - e^{-2n\delta^2} - e^{-2} + e^2 + e^{-2n\delta^2}} \\
&=& \dfrac{e^2 + e^{-2} + e^{2n\delta^2} + e^{-2n\delta^2}}{e^2 + e^{-2} + e^{2n\delta} + e^{-2n\delta}} \\
&= & \frac{e^{2n\delta^2}}{e^{2n\delta}} \times \left( \dfrac{e^{2 - 2n\delta^2} + e^{-2 - 2n\delta^2} + 1 + e^{-4n\delta^2}}{e^{2 - 2n\delta} + e^{-2 - 2n\delta} + 1 + e^{-4n\delta}} \right).
\end{eqnarray*}
Since $0<\delta<1$, we get

$$\lim_{n\to\infty}\dfrac{\max\{\phi_\sigma(nt-k): t\in  [0,1]\setminus (\frac{k}{n}-\delta, \frac{k}{n}+\delta) \}}{\min\{\phi_\sigma(nt-k): t\in[\frac{k}{n}, \frac{k}{n}+\delta^2]\}}=0.$$ This verifies the condition of the Theorem \ref{uniapp}. It is easy to see that $M_p(\phi_\sigma)$ is finite for $1\leq p<\infty$ and hence conditions on Theorem \ref{errorest} are also verified.
\end{example}

\subsection{Implementation Results}
In this section, we show the approximation of continuous and integrable functions by neural network operators with respect to the Lebesgue and Jacobi measures on $[0,1]\times[0,1].$

\subsubsection{Lebesgue Measure} First, we take $\rho$ as Lebesgue measure. Then, the operator $S_n^{\rho}$ takes the following form on $[0,1]\times[0,1]:$ \begin{equation*} S^\rho_nf(\X)=\displaystyle{\dfrac{\sum_{k_1=0}^n\sum_{k_2=0}^nc_{n,\beta}\,\Phi_\sigma(nx_1-k_1,nx_2-k_2)}{\sum_{k_1=0}^n\sum_{k_2=0}^n\Phi_\sigma(nx_1-k_1,nx_2-k_2)}}, \end{equation*} where the coefficient $c_{n,\beta}$ is given by 
$$c_{n,\beta}:=\dfrac{\int_0^1\int_0^1f(t_1,t_2)\,\Phi_\sigma(nt_1-k_1,nt_2-k_2)\,dt_1dt_2}{\int_0^1\int_0^1\Phi_\sigma(nt_1-k_1,nt_2-k_2,)\,dt_1dt_2}.$$

Now, we approximate the following continuous function: $$f(x, y) = \sin(\pi x) \cdot \cos(\pi y) + 0.5 \, x^{2} y,\quad x,y\in[0,1]\times[0,1]$$ 
by the NN operator $S^\rho_n$ with hyperbolic tangent and logistic activation function for $n=40$. The function and its approximations are given in \autoref{conti}, \autoref{tanhcont} and \autoref{logiscont}. The sup norm and the $L^1$-norm error with respect to different values of $n$ are provided in \autoref{tablogis} and \autoref{tabtanh}.

\begin{figure}[htbp]
\centering
{\includegraphics[width=0.5\textwidth]{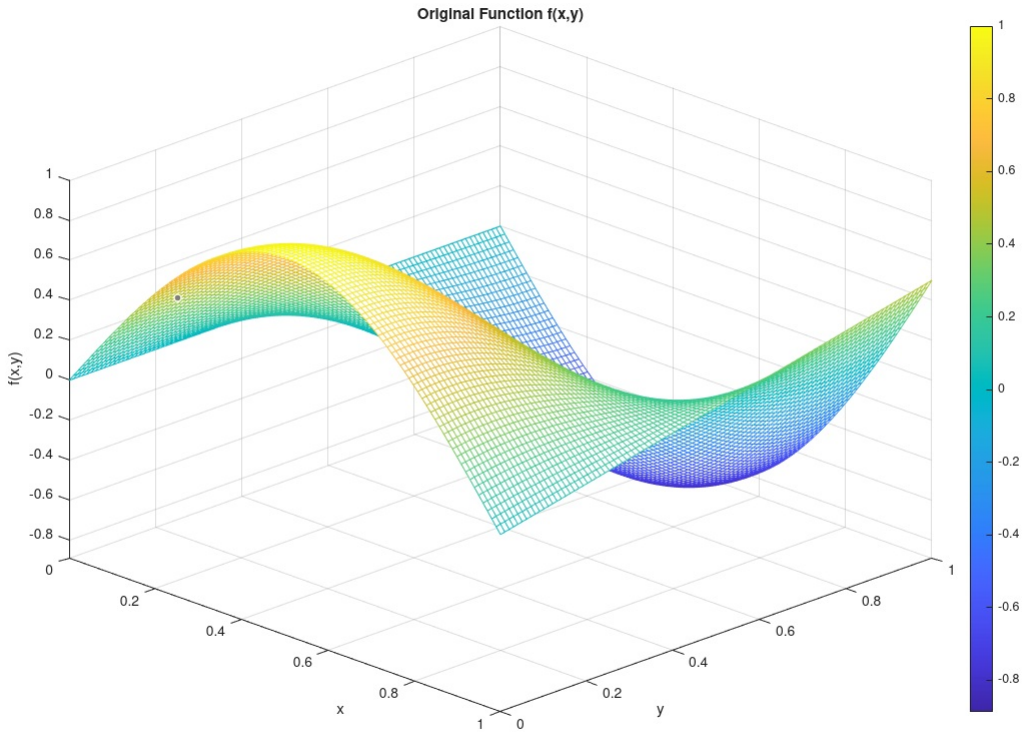} }
\caption{The original function $f(x,y)$.}
\label{conti}
\end{figure}

\begin{figure}[htbp]
\centering
{\includegraphics[width=0.5\textwidth]{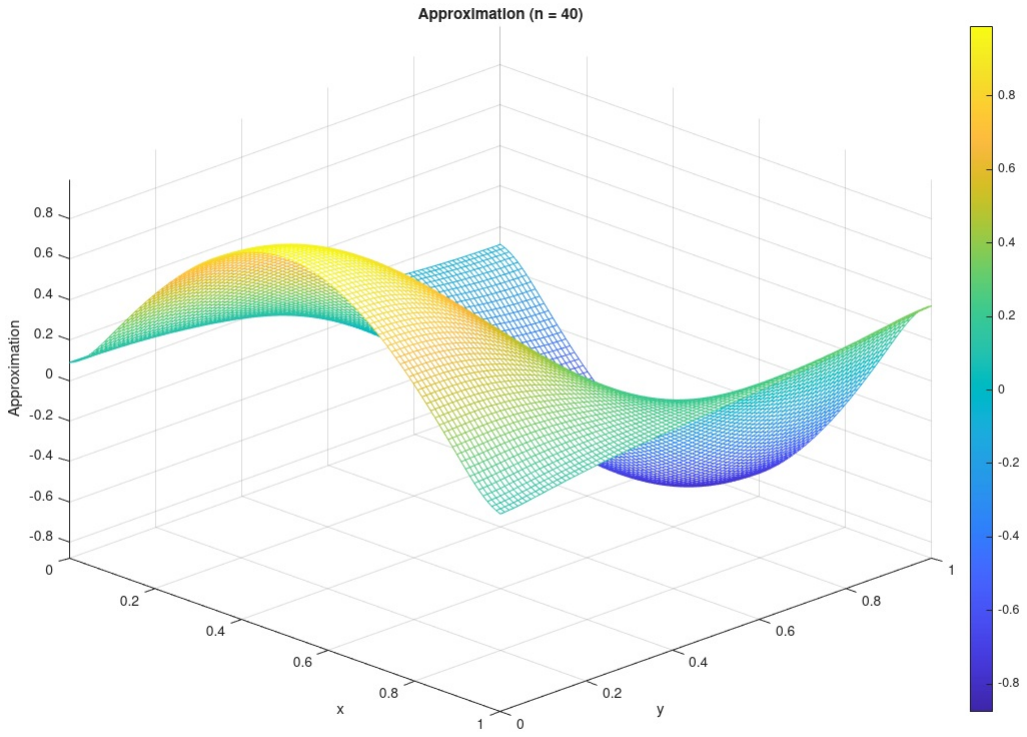} }
\caption{Approximation of $f(x,y)$ by $S^\rho_{40}f(x,y)$ with tanh activation function.}
\label{tanhcont}
\end{figure}

\begin{figure}[htbp]
\centering
{\includegraphics[width=0.5\textwidth]{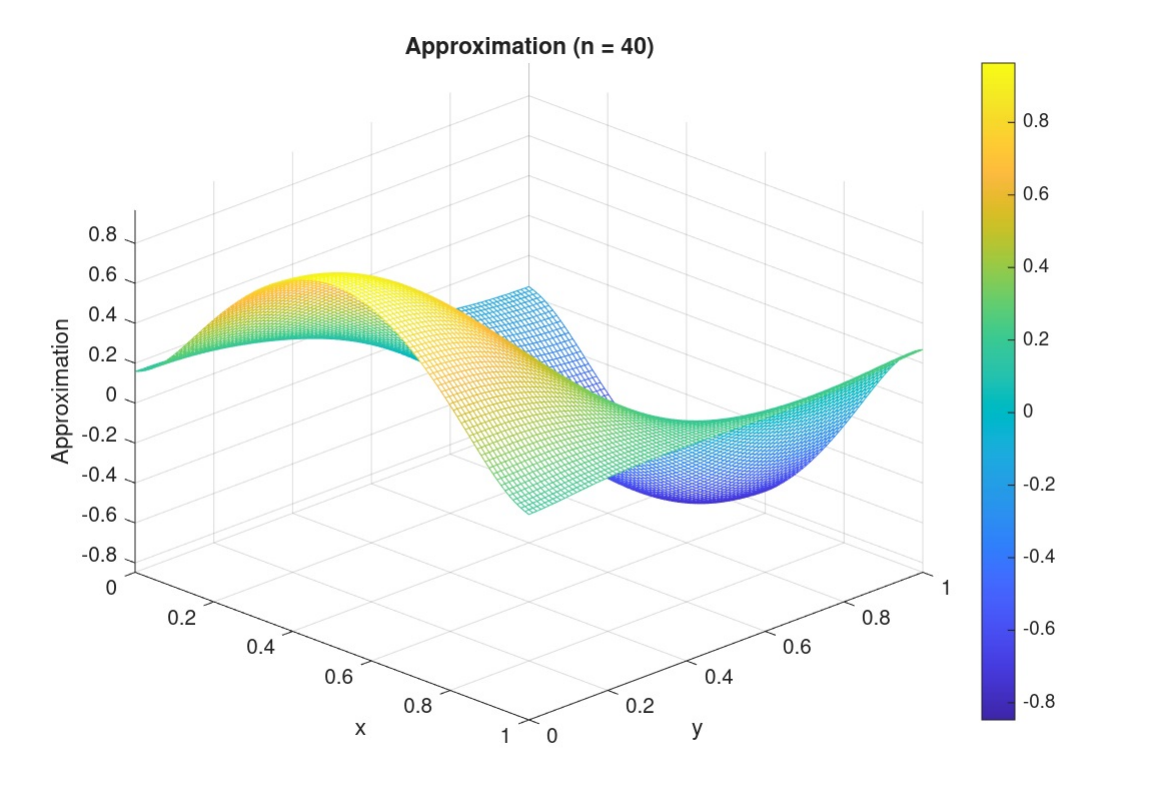} }
\caption{Approximation of $f(x,y)$ by $S^\rho_{40}f(x,y)$ with logistic activation function.}
\label{logiscont}
\end{figure}

\begin{table}[H]
\centering
\begin{tabular}{|c|c|c|}
\hline
\textbf{$n$} & \textbf{$\|S^\rho_nf-f\|_\infty$} & \textbf{$\|S_n^\rho f-f\|_{L^1([0,1]\times[0,1])}$} \\
\hline
10 & 0.6140847 & 0.18006860 \\
\hline
20 & 0.4217103 & 0.07929577 \\
\hline
40 & 0.2318002 & 0.02551001 \\
\hline
60 & 0.1577081 & 0.01215602 \\
\hline
80 & 0.1192026 & 0.00706165 \\
\hline
100 & 0.09571101 & 0.00460772 \\
\hline
120 & 0.07990336 & 0.00324348 \\
\hline
140 & 0.06854284 & 0.00240865 \\
\hline
160 & 0.05998377 & 0.00186214 \\
\hline
180 & 0.05330224 & 0.00148396 \\
\hline
\end{tabular}
\caption{Sup norm and $L^1$-norm error for varying values of $n$ with logistic activation function.}
\label{tablogis}
\end{table}

\begin{table}[H]
\centering
\begin{tabular}{|c| c| c|}
\hline
 \textbf{$n$} & \textbf{$\|S^\rho_nf-f\|_\infty$} & \textbf{$\|S^\rho_nf-f\|_{L^1([0,1]\times[0,1])}$} \\
\hline
10 & 0.4537 & 0.0936 \\
\hline
20 & 0.2536 & 0.0312 \\
\hline
40 & 0.1310 & 0.0088 \\
\hline
60 & 0.0879 & 0.0040 \\
\hline
80 & 0.0660 & 0.0023 \\
\hline
100 & 0.05277893 & 0.00150571 \\
\hline
120 & 0.04389374 & 0.00106091 \\
\hline
140 & 0.03751111 & 0.00078993 \\
\hline
160 & 0.03269716 & 0.00061305 \\
\hline
180 & 0.02893381 & 0.00049172 \\
\hline
\end{tabular}
\caption{Sup norm and $L^1$-norm error for varying values of $n$ with tanh activation function.}
\label{tabtanh}
\end{table}

Now we take the following integrable function on $[0,1]\times[0,1].$
$$
f(x, y) =
\begin{cases}
1 - 2xy, & \text{if } x < 0.4 \text{ and } y < 0.4, \\
0.3, & \text{if } 0.4 \leq x < 0.7 \text{ and } 0.4 \leq y < 0.7, \\
\sin(4\pi x) \cos(4\pi y), & \text{if } x \geq 0.7 \text{ or } y \geq 0.7.
\end{cases}
$$ 
The function and its approximation by the NN operator $S^\rho_n$ with hyperbolic tangent and logistic activation function for $n=40$ are given in \autoref{org}, \autoref{tanh40} and \autoref{logis40}. The $L^1$-norm error with respect to different values of $n$ are provided in \autoref{intt1} and \autoref{intt2}.
     
 \begin{figure}[htbp]
\centering
{\includegraphics[width=0.5\textwidth]{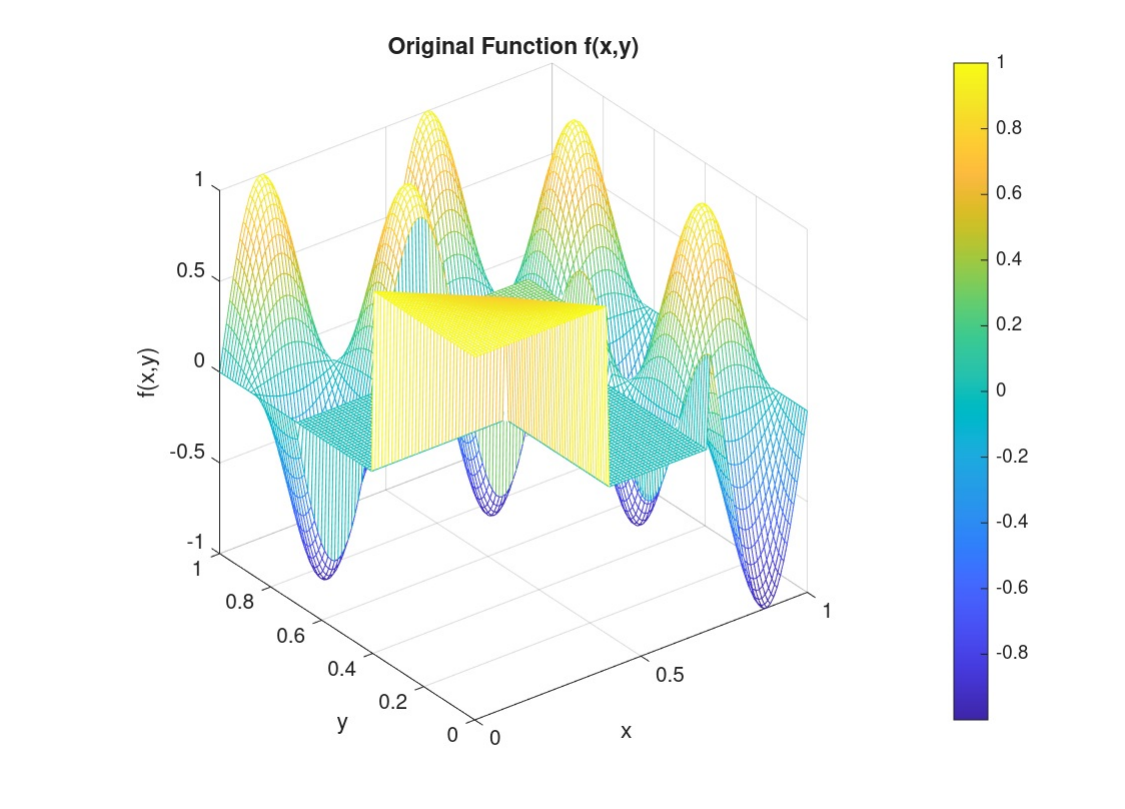} }
\caption{The original function $f(x,y)$.}
\label{org}
\end{figure}

 \begin{figure}[htbp]
\centering
{\includegraphics[width=0.5\textwidth]{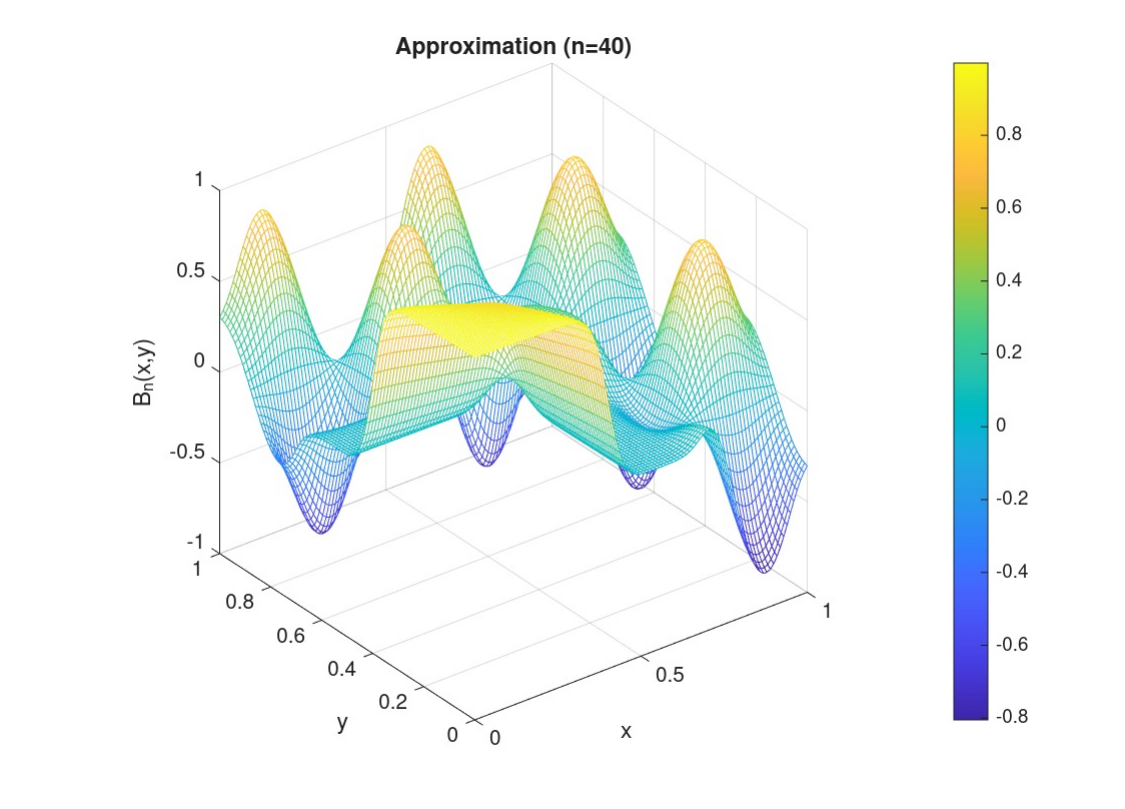} }
\caption{Approximation of $f(x,y)$ by $S^\rho_{40}f(x,y)$ with tanh activation function.}
\label{tanh40}
\end{figure}

 \begin{figure}[htbp]
\centering
{\includegraphics[width=0.5\textwidth]{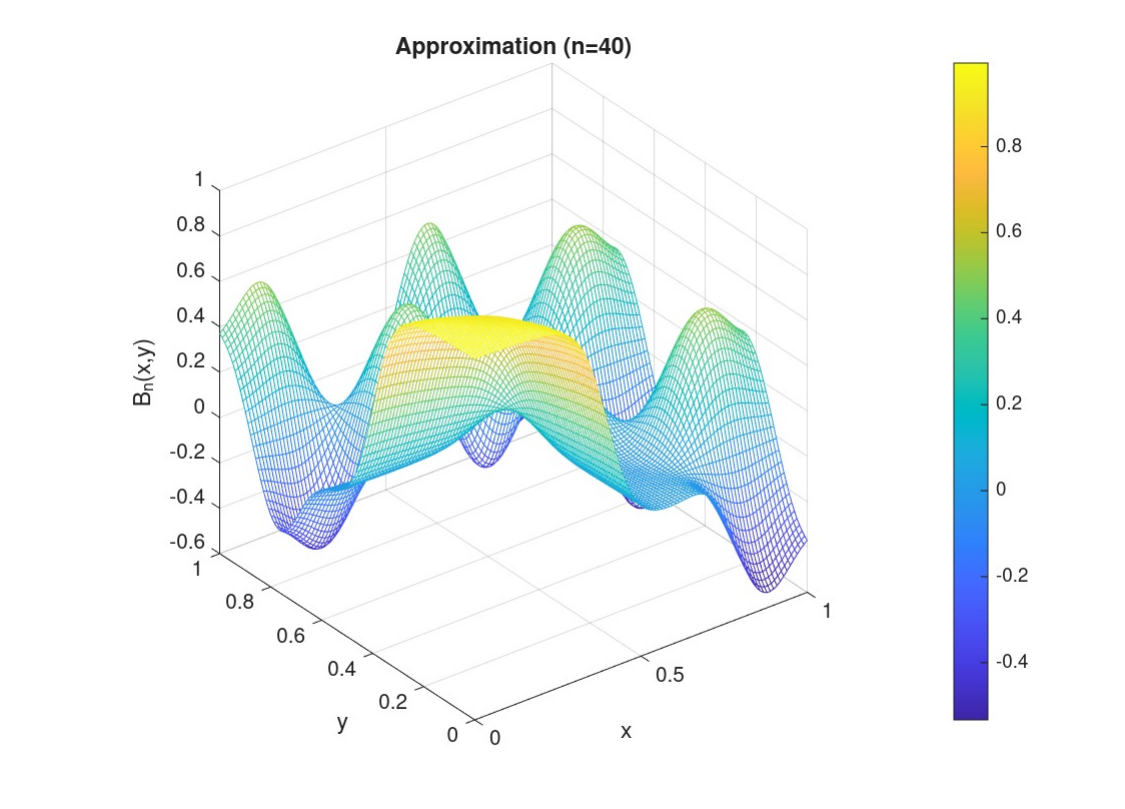} }
\caption{Approximation of $f(x,y)$ by $S^\rho_{40}f(x,y)$ with logistic activation function.}
\label{logis40}
\end{figure}

\begin{table}[H]
\centering
\begin{tabular}{|c|c|c|}
\hline
\textbf{$n$} & \textbf{$\|S^\rho_nf-f\|_{L^1([0,1]\times[0,1])}$} \\ 
\hline
10   & 0.34143 \\
\hline
20   & 0.26699 \\
\hline
40    & 0.15767 \\
\hline
60    & 0.10089 \\
\hline
80    & 0.07069 \\
\hline
100 & 0.00460772 \\
\hline
120 & 0.00324348 \\
\hline
140 & 0.00240865 \\
\hline
160 & 0.00186214 \\
\hline
180 & 0.00148396 \\
\hline
\end{tabular}
\caption{$L^1$ norm error for varying values of $n$ with logistic activation function.}
\label{intt1}
\end{table}

\begin{table}[H]
\centering
\begin{tabular}{|c|c|}
\hline
\textbf{$n$} & \textbf{$\|S^\rho_nf-f\|_{L^1([0,1]\times[0,1])}$} \\
\hline
10  & 0.28442 \\
\hline
20  & 0.17863 \\
\hline
40  & 0.08282 \\
\hline
60  & 0.04929 \\
\hline
80  & 0.03390 \\
\hline
100 & 0.02538742 \\
\hline
120 & 0.02004754 \\
\hline
140 & 0.01640804 \\
\hline
160 & 0.01377641 \\
\hline
180 & 0.01179791 \\
\hline
\end{tabular}
\caption{$L^1$ norm error for varying values of $n$ with tanh activation function.}
\label{intt2}
\end{table}

\subsubsection{Jacobi Measure} 
Now, we consider the following Jacobi weight measure: $$w(t_1,t_2)=t_{1}^{\alpha}(1-t_{1})^{\beta} \, t_{2}^{\gamma}(1-t_{2})^{\delta},$$ where $\alpha=\beta= \gamma= \delta=0.5$ and $t_1, t_2\in [0,1].$ For the Jacobi weight, the operator $S_n^{\rho}$ takes the following form in $[0,1]\times [0,1]$:
\begin{equation}
    S^w_nf(\X)=\displaystyle{\dfrac{\sum_{k_1=0}^n\sum_{k_2=0}^nc_{n,\beta}\,\Phi_\sigma(nx_1-k_1,nx_2-k_2)}{\sum_{k_1=0}^n\sum_{k_2=0}^n\Phi_\sigma(nx_1-k_1,nx_2-k_2)}},\end{equation} where the coefficient $c_{n,\beta}$ is given by 
$$c_{n,\beta}:=\dfrac{\int_0^1\int_0^1f(t_1,t_2)\,\Phi_\sigma(nt_1-k_1,nt_2-k_2)\,w(t_1,t_2)\,dt_1dt_2 
}{\int_0^1\int_0^1\Phi_\sigma(nt_1-k_1,nt_2-k_2)\,w(t_1,t_2)\,dt_1dt_2 }.$$ 

Now we approximate the following integrable function by NN operators $S_n^w$:
$$ f(x, y) =
\begin{cases}
1 - 2xy, & \text{if } x < 0.4 \text{ and } y < 0.4, \\
0.3, & \text{if } 0.4 \leq x < 0.7 \text{ and } 0.4 \leq y < 0.7, \\
\sin(4\pi x) \cos(4\pi y), & \text{if } x \geq 0.7 \text{ or } y \geq 0.7.
\end{cases} $$ 

The function and its approximation by the NN operator $S^w_n$ with hyperbolic tangent and logistic activation function for $n=40$ are given in \autoref{org2}, \autoref{jacoblog} and \autoref{jacobitanh}. The $L^1$-norm error with respect to different values of $n$ is provided in \autoref{t3} and \autoref{t4}.

\begin{figure}[htbp]
\centering
{\includegraphics[width=0.5\textwidth]{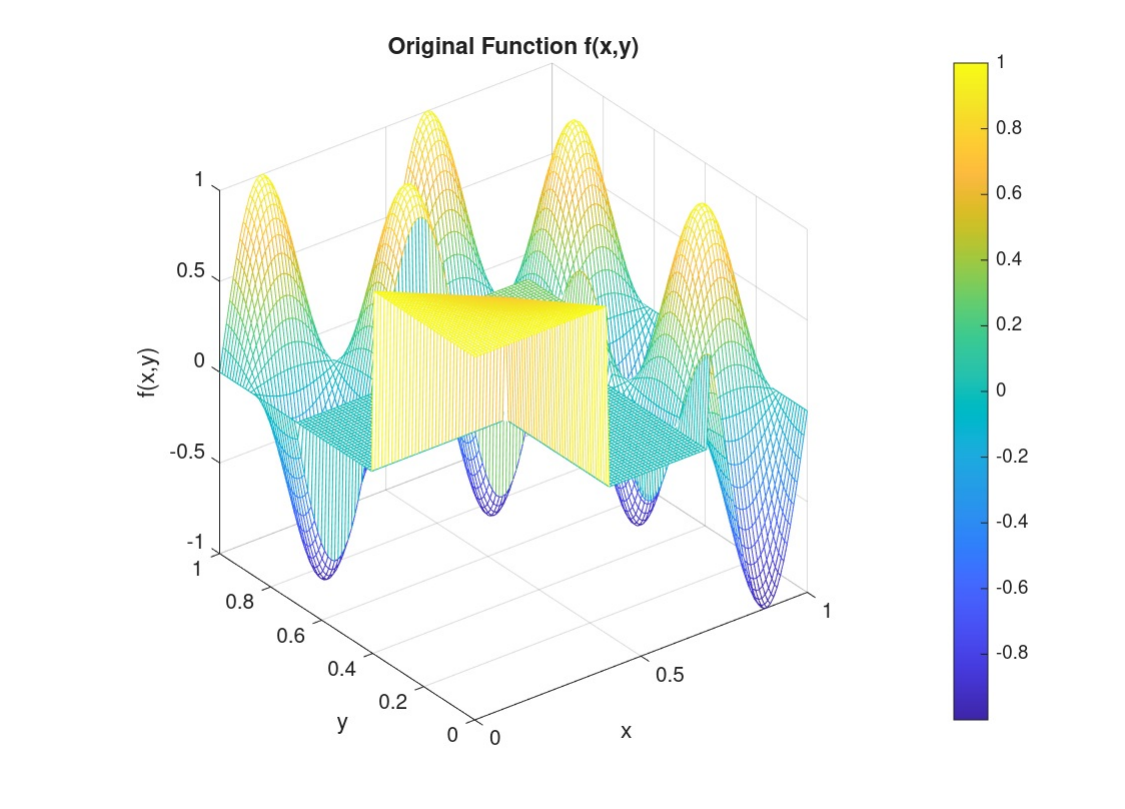} }
\caption{The original function $f(x,y)$.}
\label{org2}
\end{figure}

\begin{figure}[htbp]
\centering
{\includegraphics[width=0.5\textwidth]{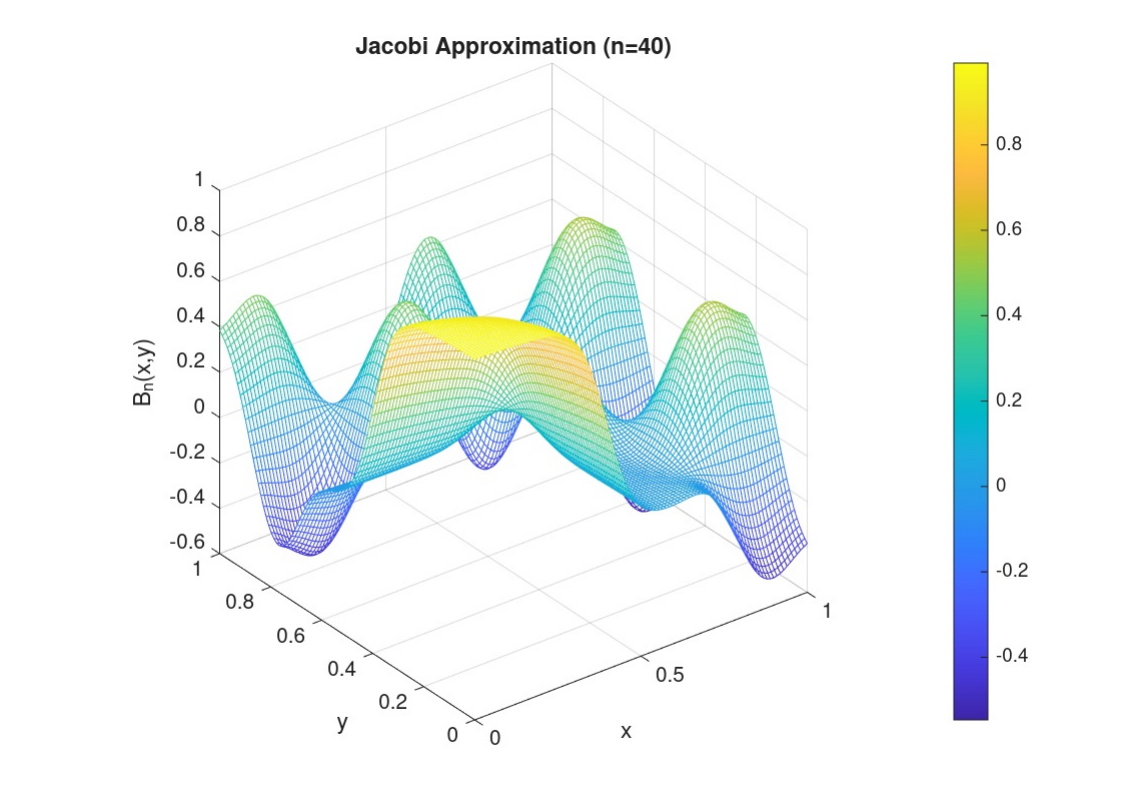} }
\caption{Approximation of $f(x,y)$ by $S^w_{40}f(x,y)$ with logistic activation function.}
\label{jacoblog}
\end{figure}

\begin{figure}[htbp]
\centering
{\includegraphics[width=0.5\textwidth]{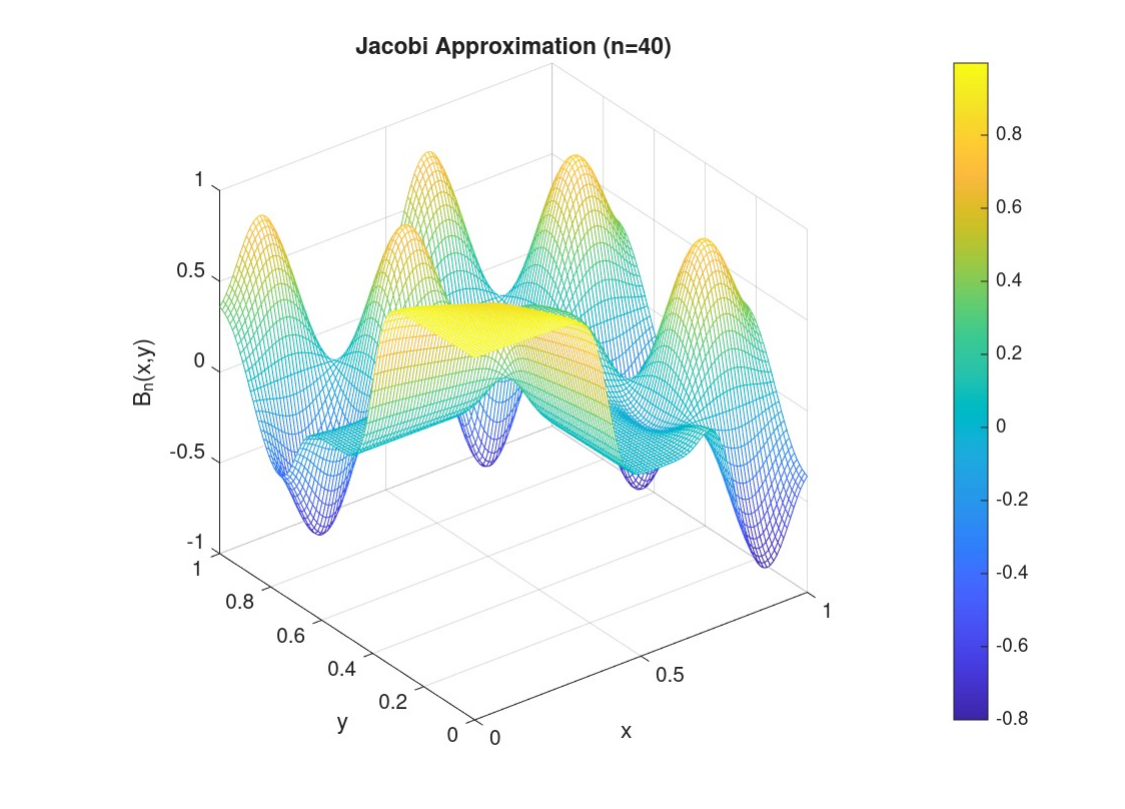} }
\caption{Approximation of $f(x,y)$ by $S^w_{40}f(x,y)$ with tanh activation function.}
\label{jacobitanh}
\end{figure}

\begin{table}[H]
\centering
\begin{tabular}{|c| c|}
\hline
$n$ & $\|S^w_nf-f\|_{L^1([0,1]\times[0,1])}$ \\
\hline
10 & 0.050147 \\
\hline
20 & 0.040254 \\
\hline
40 & 0.024157 \\
\hline
60 & 0.015721 \\
\hline
80 & 0.011188 \\
\hline
100 & 0.008522 \\
\hline
120 & 0.006810 \\
\hline
140 & 0.005629 \\
\hline
160 & 0.004773 \\
\hline
180 & 0.004126 \\
\hline
\end{tabular}
\caption{$L^1$ norm errors for different values of $n$ with logistic activation function.}
\label{t3}
\end{table}

\begin{table}[H]
\centering
\begin{tabular}{|c| c|}
\hline
$n$ &  $\|S^w_nf-f\|_{L^1([0,1]\times[0,1])}$ \\
\hline
10 & 0.042775 \\
\hline
20 & 0.027303 \\
\hline
40 & 0.013046 \\
\hline
60 & 0.007969 \\
\hline
80 & 0.005590 \\
\hline
100 & 0.004245 \\
\hline
120 & 0.003391 \\
\hline
140 & 0.002803 \\
\hline
160 & 0.002371 \\
\hline
180 & 0.002040 \\
\hline
\end{tabular}
\caption{$L^1$ norm errors for different values of $n$ with tanh activation function.}
\label{t4}
\end{table}

\section{Final Remarks and Conclusions}
\subsection{Final Remarks} We have the following concluding remarks.
\begin{itemize}
    \item In this paper, we have considered the unit hypercube $[0,1]^d\subset \R^d$. It is easy to see that the similar results are also applicable to more general sets $\Omega \subset \R^d$, where $\Omega:=\displaystyle\prod_{i=1}^{d}[a_i, b_i].$ Hence, it is not necessary to repeat the details.\\

   \item We have verified the hypothesis of the theorems for the logistic and hyperbolic tangent activation functions. It would be interesting to look for other sigmoidal functions that satisfy the hypothesis of the theorem. \\


    \item It would be insightful to study the operator (\ref{mopdef}) for specific weighted measure and see how the choice of weight influences the convergence properties of the operator.
    
   \end{itemize}
    
\subsection{Conclusions}
The approximation of functions belonging to the $L^p(I^d,\rho),$ where $1\leq p<\infty$ is associated with an arbitrary measure $\rho$ defined on a hypercube satisfying a certain support condition by the NN operators is investigated. 
Specifically, the uniform approximation of continuous functions defined on a hypercube by these operators is proved. Further, the $L^p(I^d,\rho)$ approximation and its error rate in terms of $\mathcal{K}-$functional is obtained. 
Towards the end, the hypothesis of the theorems are verified for the logistic and hyperbolic tangent activation functions. The approximation of particular continuous and integrable functions by NN operators with respect to the Lebesgue and Jacobi measures defined on $[0,1]\times[0,1]$ with these activation functions has been shown.

\vspace{3mm}
\subsubsection*{\bf Data availability:}
\noindent No data was used for the research described in the article.

 \subsection*{Acknowledgements}

\noindent
 Nitin Bartwal is thankful to the HTRA Fellowship, IIT Madras for the financial support to carry out his research work.
\noindent 
A. Sathish Kumar is supported by ANRF, DST-SERB, India Research Grant: EEQ/2023/000257 for financial support.

\end{document}